\documentclass[11pt,oneside]{amsart}
\usepackage[euler-digits]{eulervm}
\usepackage{upgreek}
\usepackage[leqno]{amsmath}
\usepackage{amsfonts, amsthm, amssymb, stmaryrd}
\usepackage{mathrsfs,array}
\usepackage{fullpage,times,color,enumerate,accents}
\usepackage{url}
\usepackage{dsfont}
\usepackage[euler]{textgreek}
\usepackage{fullpage, tikz, cmll, epic, mathtools} 
\usepackage{scrextend}

\newtheorem{innercustomthm}{Theorem}
\newenvironment{customthm}[1]
  {\renewcommand\theinnercustomthm{#1}\innercustomthm}
  {\endinnercustomthm}
 
\usepackage{color}
\usepackage{mathrsfs}
\usepackage{amssymb}
\usepackage{tikz}
\usepackage{tikz-cd}
\usepackage{bm}
\usepackage{enumerate}
\usetikzlibrary{calc}
\usetikzlibrary{fadings}
\usetikzlibrary{decorations.pathmorphing}
\usetikzlibrary{decorations.pathreplacing}
\usepackage{xcolor}

\DeclareMathAlphabet{\mathpzc}{OT1}{pzc}{m}{it}

\title{{\larger S}keletons of stable maps I: Rational curves in toric varieties}
\author{\vspace{-0.0in}{{\larger D}{\smaller hruv}\ \ {\larger R}{\smaller anganathan}}}
\date{\today}

\usepackage[backref]{hyperref}
\hypersetup{
  colorlinks   = true,          
  urlcolor     = blue,          
  linkcolor    = blue,          
  citecolor   = purple             
}

\address{Department of Mathematics, Massachusetts Institute of Technology, Cambridge, MA 02138}
\email{dhruvr@mit.edu}

\newtheorem{theorem}{Theorem}[subsection]

\newtheorem{lemma}[theorem]{Lemma}
\newtheorem{proposition}[theorem]{Proposition}
\newtheorem{definition}[theorem]{Definition}

\newtheorem{quasi-theorem}[theorem]{Quasi-Theorem}
\newtheorem{blank remark}[theorem]{}

\newtheorem{rem1}[theorem]{Remark}
\newenvironment{remark}{\begin{rem1}\em}{\end{rem1}}

\newtheorem{not1}[theorem]{Notation}

\newcommand{\CC} {{\mathbb C}}          
            
\newcommand{\NN} {{\mathbb N}}		
\newcommand{\PP}{\mathbb{P}}         
\newcommand{\QQ} {{\mathbb Q}}		
\newcommand{\RR} {{\mathbb R}}		
\newcommand{\ZZ} {{\mathbb Z}}

\newcommand{\Hom}{\operatorname{Hom}}

\DeclareMathOperator{\val}{val}

\DeclareMathOperator{\spec}{Spec}

\newcommand{\cal}{\mathcal}

\def\cE{{\cal E}}
\def\cF{{\cal F}}

\def\cM{{\cal M}}

\def\cT{{\cal T}}

\def\cX{{\cal X}}

\def\fM{\mathfrak{M}}
\def\fS{\mathfrak{S}}

\newcommand{\Mbar}{\overline{\cM}}

\newcommand{\Lsm}{\mathcal{LSM}}
\newcommand{\Tsm}{{TSM}}

\def\trop{\mathrm{trop}}
\def\an{\mathrm{an}}

\makeatletter
\def\blfootnote{\xdef\@thefnmark{}\@footnotetext}
\makeatother

\begin{document}

\pagestyle{plain}
\vspace{-1in}
\maketitle

\begin{abstract}
\vspace{-0.3in}
We study the Berkovich analytification of the space of genus $0$ logarithmic stable maps to a toric variety $X$ and present applications to both algebraic and tropical geometry. On algebraic side, insights from tropical geometry give two new geometric descriptions of this space of maps -- (1) as an explicit toroidal modification of $\overline M_{0,n}\times X$ and (2) as a tropical compactification in a toric variety. On the combinatorial side, we prove that the tropicalization of the space of genus $0$ logarithmic stable maps coincides with the space of tropical stable maps, giving a large new collection of examples of faithful tropicalizations for moduli. Moreover, we identify the optimal settings in which the tropicalization of the moduli space of maps is faithful. The Nishinou--Siebert correspondence theorem is shown to be a consequence of this geometric connection between the algebraic and tropical moduli.
\end{abstract}

\blfootnote{This research was partially supported by funds from NSF grant CAREER DMS-1149054 (PI: Sam Payne).}

\section{Introduction}

The primary objective of this paper is to explore the interplay between the algebraic, tropical, and non-archimedean geometry of the space of logarithmic stable maps to a toric variety in genus $0$. There are three outcomes of this study, which we catalogue before stating the results formally. The first of these is algebro-geometric in nature, the second is a new result in the realm of faithful tropicalization, and the third is a new perspective on a fundamental correspondence theorem in tropical enumerative geometry.\vspace{-3.5mm}

\subsection*{I. The structure of the space of stable maps} The space $\Lsm(X)$ of logarithmic stable maps to a projective toric variety $X$ compactifies the space of geometric genus $0$ curves in $X$ with prescribed contact orders with toric boundary divisors. This space is given a new and concrete description as a toroidal modification of $\overline M_{0,n}\times X$. Moreover, its logarithmic structure is shown to have a particularly simple form: it arises as the pullback of the logarithmic structure on a toric variety via a natural embedding. See Theorem~\ref{thm: trop-comp}.\vspace{-3.5mm}

\subsection*{II. Faithful tropicalization for moduli spaces} The space $TSM(\Delta)$ of tropical stable maps to a the fan $\Delta$ of $X$ is canonically identified with the skeleton of the Berkovich analytification of $\Lsm(X)$. This provides a new infinite family of examples of this phenomenon, building upon~\cite{ACP} and directly generalizing~\cite{CMR14b, Tev07} which prove this in the special cases when $X$ is a point or $\PP^1$. Moreover, our result is essentially the sharpest possible one of this form -- away from trivial cases, tropicalizations of moduli spaces of higher genus stable maps are never faithful. Indeed, this is already the case for genus $1$ maps to $\PP^1$, see~\cite{CMR14a}. This is an important technical ingredient in the above result, see Theorem~\ref{thm: skeleton}.\vspace{-3.5mm}

\subsection*{III. Correspondence theorems for moduli spaces} In~\cite{CMR14a,CMR14b}, Cavalieri, Markwig, and the author initiated a program to understand algebraic/tropical correspondence theorems for enumerative invariants by proving ``geometrized'' correspondence theorems at the level of algebraic/tropical moduli spaces. The present paper achieves this for genus $0$ enumerative invariants of toric varieties, recovering, in a simple fashion, the celebrated correspondence theorem of Nishinou and Siebert~\cite{NS06}. See Theorem~\ref{thm: enumerative}.

\subsection{Formal statement of results}

Let $X$ be a projective toric variety with fan $\Delta$ and $\Lsm_\Gamma(X)$ denote the moduli space of genus $0$ logarithmic stable maps $[f:C\to X]$ having curve class $\beta$, and contact order $c$ to the toric boundary divisor along $n$ marked points. This space was constructed in the papers~\cite{AC11,Che10,GS13}. We package the discrete data $(\beta,n,c)$ in the symbol $\Gamma$. In Section~\ref{lsm-skeleton} we describe an extended cone complex $\Tsm_\Gamma(\Delta)$ parametrizing tropical maps to the compactified fan $\overline \Delta$ with discrete data $\Gamma$. We then construct a set-theoretic tropicalization map
\[
\trop:\Lsm^{\an}_\Gamma(X)\to \Tsm_\Gamma(\Delta).
\] 

\begin{customthm}{A}\label{thm: skeleton}
There is a continuous deformation retraction $\bm p: \Lsm_\Gamma^{\an}(X)\to \overline \fS$ projecting $\Lsm_\Gamma^{\an}(X)$ onto a skeleton, and an isomorphism $\trop_\fS: \overline \fS \to \Tsm_\Gamma(\Delta)$ of extended cone complexes with integral structure making the following diagram commute
\[
\begin{tikzcd}
\Lsm^{\an}_\Gamma(X) \arrow[swap]{dr}{\bm p} \arrow{rr}{\trop} & & \Tsm_\Gamma(\Delta) \\
& \overline \fS \arrow[swap]{ur}{\trop_{\fS}}. & \\
\end{tikzcd}
\]
\end{customthm}

The result above also forms the main technical ingredient in the following. Consider the moduli space of geometric genus $0$ curves $C$ in $X$ meeting the dense torus, with a fixed contact order with each toric boundary divisor. Assume that there is at least $1$ marked point of $C$ that maps to the dense torus of $X$. Denote this space by $\Lsm^\circ_\Gamma(X)$. Let $TSM_\Gamma^\circ(\Delta)$ denote interior of the extended cone complex $TSM_\Gamma(\Delta)$.

\begin{customthm}{B}\label{thm: trop-comp}
Let $\Delta_{0,n}$ denote the fan of the toroidal embedding $\overline M_{0,n}\times X$. Then,
\begin{enumerate}
\item There is a natural refinement of cone complexes
\[
TSM^\circ_\Gamma(\Delta)\to \Delta_{0,n},
\]
such that the associated toroidal modification of $\overline M_{0,n}\times X$ is isomorphic to $\Lsm_\Gamma(X)$. 
\item There exists torus $T$ an embedding $\Lsm_\Gamma^\circ(X)\hookrightarrow T$, such that the closure of $\Lsm_\Gamma^\circ(X)$ in the toric variety defined by $TSM_\Gamma(X)$ coincides with the coarse space of $\Lsm_\Gamma(X)$. 
\end{enumerate}
\end{customthm}

In colloquial terms, the second part of the result above states that $\Lsm_\Gamma^\circ(X)$ and the tropical moduli space $TSM_\Gamma(X)$ together determine the compactification $\Lsm_\Gamma(X)$. An important ingredient in the proof of this result is the irreducibility of the moduli space of genus $0$ logarithmic stable maps to a toric variety, which we prove in Proposition~\ref{prop: irred}, and may be of independent interest. 

%

The simplest case where the features of the results are visible is the moduli space of lines in $\PP^2$ with contact order $1$ along marked points (the ``logarithmic'' dual $\PP^2$). In Section~\ref{sec: extended-example} we work out this case by hand, proving a toy version of the theorem above in Theorem~\ref{toytheorem}.

Finally, we explain how to re-derive the correspondence theorem of Nishinou and Siebert~\cite[Theorem 8.3]{NS06} in this setting. Let $N$ be the cocharacter lattice of the dense torus of $X$. Let $\mathscr L = (L_1,\ldots, L_m)$ be an $m$-tuple of linear subspaces of $N_\QQ$. Let $Z_i$ be the closure of the associated subtorus $\mathbb G(L_i)$. This data defines a logarithmic Gromov--Witten invariant
\[
\langle Z_1,\ldots, Z_m\rangle^{X}_{\Gamma} := \int_{[\Lsm_\Gamma(X)]^{\mathrm{vir}}} ev_1^\star [Z_1] \wedge \cdots \wedge ev_m^\star [Z_m].
\]
In this case, the virtual fundamental class agrees with the usual fundamental class, and the invariants are genuine counts of rational curves in $X$ with incidence to subtorus closures, see Proposition~\ref{prop: logsmooth}. Let $\overline \Delta_{L_i}$ denote the compactified fan of the Chow quotient $X\!\sslash\!\mathbb G(L_i)$ in the sense of~\cite[Section 1]{KSZ91}.

\begin{customthm}{C}\label{thm: enumerative}
There is a natural surjective map of cone complexes with integral structure 
\[
Ev_{\mathscr L}^\trop: \Tsm_\Gamma(\Delta)\to \prod\overline \Delta_{L_i}.
\] 
whose degree on each maximal cell of $\prod \overline \Delta_{L_i}$ is constant, and
\[
\deg(Ev^\trop_{\mathscr L}) = \langle Z_1,\ldots, Z_m\rangle^{X}_{\Gamma}.
\]
\end{customthm}

\noindent
Here, by \textit{degree}, we mean the sum of dilation factors of $Ev^{\trop}_{\mathscr L}$, upon restriction to top dimensional cells of $\Tsm_\Gamma(\Delta)$ that map to a fixed top dimensional cell of $\prod \overline \Delta_{L_i}$. It is computed as a sum of lattice indices over combinatorial types, as explained in~\cite[Theorem 5.1]{GKM07}.

Our proof of Theorem~\ref{thm: enumerative} is based on the Sturmfels--Tevelev multiplicity formula~\cite[Theorem 1.1]{ST}, and the identification of $\Tsm_\Gamma(\Delta)$ as a skeleton of $\Lsm_\Gamma(X)$. This proof differs substantially from the original~\cite[Section 8]{NS06} as it does not involve a choice of a toric degeneration of $X$. In loc. cit., a tropical stable map in $\Delta$ corresponds to a degeneration of $X$ to a broken toric variety $X_0$, together with a nodal curve in $X_0$ that is transverse to the toric strata. In this paper, given a point $p\in \Tsm_\Gamma(\Delta)$, the set $\trop^{-1}(p)$ is an analytic affinoid domain in $\Lsm^{\an}_\Gamma(X)$ that parametrizes stable maps to $X$ with prescribed tropicalization, rather than maps to a degeneration of $X$. 

\subsection{Related results and future directions} The result of Theorem~\ref{thm: skeleton} builds on earlier work of Abramovich, Caporaso, and Payne~\cite[Theorem 1.2.1]{ACP}, identifying the skeleton of the moduli stack of stable pointed curves with a complex of abstract tropical curves. It ultimately relies on techniques developed by Thuillier~\cite{Thu07}. There are variations on this result for toroidal compactifications of spaces of smooth weighted pointed curves~\cite{CHMR,U14a}. For spaces of (ramified) maps, one still expects a continuous map from the analytified moduli space of maps to a tropical moduli space~\cite{Yu14b}. However this no longer identifies the tropical moduli space with a skeleton. In fact, this identification fails even for low degree maps from genus $1$ curves to $\PP^1$, even though there is a toroidal compactification by the space of admissible covers~\cite{CMR14a}. In this sense, the genus $0$ requirement in Theorem~\ref{thm: skeleton} is sharp. The tropicalization and analytification of higher genus logarithmic maps is explored in the sequel to this article~\cite{R16} using Artin fan techniques. 

The result of Theorem~\ref{thm: trop-comp} gives concrete handle on the geometry of the moduli space of logarithmic stable maps in genus $0$. A very closely related result appears in Chen and Satriano's description of this space in the special case where the general curve is the closure of a one-parameter subgroup in $X$. See~\cite[Theorem 1.1]{CS12}. Ascher and Molcho have recently generalized these results to higher rank subtori~\cite[Theorem 1.3]{AM14}.

Together with the results of~\cite{CMR14a,CMR14b}, Theorem~\ref{thm: enumerative} contributes to the understanding of the role of tropical computations in degeneration formulas for relative Gromov--Witten theories. The cases considered thus far are enumerative problems associated to proper moduli spaces that are toroidal. Beyond these cases, it appears that tropical curve counts will be used in unison with virtual techniques, as in the forthcoming article of Abramovich, Chen, Gross, and Siebert~\cite{ACGS15}. A key role in loc.\! cit. is 
played by the stack of pre-stable maps to the Artin fan of a logarithmic variety~\cite[Proposition 1.5.1]{AW}. The connection between such maps and tropical and non-archimedean geometry is initiated studied in~\cite{R15a} and will be further studied in~\cite{R16}. 

Since this article first appeared on the ar$\chi$iv, A. Gross has used the Theorem~\ref{thm: trop-comp} to prove algebraic/tropical descendant correspondence theorems for toric varieties~\cite{Gro15}. We expect additional applications. In future work, we intend to use this description of the space of maps to compute the cohomology of the space of maps and to study genus $0$ characteristic numbers for toric varieties. In~\cite{Pan99}, Pandharipande uses similar computations to derive recursions for characteristic numbers of projective space.

\subsection*{Acknowledgements} This paper represents Chapter IV of my dissertation at Yale University. It is a pleasure to acknowledge the ideal working conditions at Brown University in Spring 2015. I extend my gratitude to Dan Abramovich and Sam Payne for encouragement and many insightful discussions. Thanks are due to my collaborators Renzo Cavalieri and Hannah Markwig -- this project grew out of an effort to generalize the results of~\cite{CMR14a,CMR14b} to higher dimensions. I have benefited from conversations with Dori Bejleri, Qile Chen, Tyler Foster, Mark Gross, Bernd Siebert, Martin Ulirsch, and Jonathan Wise. I learned about logarithmic stable maps at the 2014 summer school on toric degenerations in Norway, and I thank Lars Halle and Johannes Nicaise for creating that opportunity. 
 The text has benefited from the careful comments of an anonymous referee.

\section{Preliminaries}\label{sec: background}

We provide a brief review of skeletons of non-archimedean analytic spaces and of logarithmic Gromov--Witten theory.
\subsection{Berkovich spaces and skeletons}\label{sec: background-1} Let $K$ be a field, complete with respect to a possibly trivial rank-$1$ valuation $\nu:K^\times\to \RR$. We assume throughout that $K$ is equicharacteristic $0$ with algebraically closed residue field. Let $X$ be a finite-type, irreducible, and separated $K$-scheme. The Berkovich analytification $X^{\an}$ of $X$ is a locally ringed space that plays a role in non-archimedean geometry that is analogous to the usual analytification of a complex variety.  As a set, $X^{\an}$ consists of maps $\spec(L)\to X$, where $(L,\nu_L)$ is a valued extension field of $(K,\nu)$, subject to the equivalence relation generated by declaring
\[
x = (\spec(L)\to X)\sim x' = (\spec(L')\to X),
\]
if $L'$ is an extension of $L$, the valuation $\nu_{L'}$ restricts to $\nu_L$ on $L$, and the scheme theoretic points underlying $x$ and $x'$ are identified under the inclusion $X(L)\hookrightarrow X(L')$. 

\begin{theorem}[{Berkovich~\cite[Chapter 4]{Ber90} and~\cite{Ber99}}]
With the assumptions on $X$ as above, the space $X^{\an}$ is a path connected, locally compact, Hausdorff topological space. Moreover, $X^{\an}$ admits a deformation retraction onto a finite-type polyhedral complex. If $X$ is proper, then $X^{\an}$ is topologically compact. 
\end{theorem}

When $K$ is the trivially valued field $\CC$, there is an alternative analytification functor denoted $(\cdot)^{\beth}$, defined by Thuillier~\cite{Thu07}. For $X$ separated, $X^\beth$ consists of the subset of points of $X^{\an}$ represented by maps $\spec(L)\to X$ that extend to $\spec(R_L)\to X$, where $R_L$ is the valuation ring of $L$. Both spaces $X^{\an}$ and $X^\beth$ are equipped with structure sheaves of analytic functions, locally given by limits of rational functions on $X$. 

\begin{theorem}[{Thuillier~\cite{Thu07}}]
The space $X^\beth$ is a compact domain in $X^{\an}$ and admits a deformation retraction onto a finite-type polyhedral complex. 
\end{theorem}

The functors $(\cdot)^{\an}$ and $(\cdot)^\beth$ extend to algebraic stacks of locally finite type, see~\cite{U14b,Yu14a} and~\cite[Section V.3]{U-thesis}. The retractions of analytic spaces onto polyhedral complexes or \textit{skeletons} include and generalize tropicalization for toric varieties. We briefly review this, and refer the reader to the survey~\cite{ACMUW} for a more complete discussion and references.

Let $M$ be a lattice with dual lattice $N$, and let $T = \spec(K[M])$ be the associated algebraic torus. Each point of $p \in T^{\an}$ gives rise to a morphism
\[
\val_p: K[M]\to L\xrightarrow{\nu_L} \RR\sqcup \{\infty\}.
\]
Since the valuation $\nu_L$ is multiplicative, restriction of $\val_p$ to $M$ furnishes a homomorphism of abelian groups $\trop(\val_p)\in\Hom_{\bf Ab}(M,\RR) \cong N_\RR$. This defines the \textit{tropicalization} map 
\[
\trop: T^{\an}\to N_\RR. 
\]

Let $X$ be a proper toric variety over a non-archimedean field associated to a complete fan $\Delta$ in $N_\RR$. Denote by $\overline \Delta$ the associated canonical compactification of $\Delta$, obtained as follows. Given a cone $\sigma\in \Delta$, let $S_\sigma$ be the dual monoid. Define the extended cone compactifying $\sigma$ by
\[
 \overline \sigma := \Hom_{\bf Mon}(S_\sigma,\RR_{\geq 0}\sqcup \{\infty\}).
\] 
The topology on $\RR_{\geq 0}\sqcup\{\infty\}$ is determined by the extended order topology, declaring $r<\infty$ for all $r\in \RR_{\geq 0}$. The set $\overline \sigma$ is given the topology of pointwise convergence. The extended cones are glued in the natural way to form an extended cone complex $\overline \Delta$ compactifying $N_\RR$. We refer to~\cite[Section 2]{ACP} and~\cite[Section 3]{Thu07} for further details on extended cone complexes and generalizations.

There is a continuous proper surjection
\[
\trop: X^{\an}\to \overline \Delta,
\]
that restricts to the tropicalization map defined above on the analytic torus $T^{\an}\subset X^{\an}$. Its image is referred to as the \textit{extended tropicalization of $X$}. The compactified fan $\overline \Delta$ is naturally stratified into vector spaces $N(\sigma) = N_\RR/\mathrm{span}(\sigma)$. These vector spaces $N(\sigma)$ are the images of locally closed strata $V(\sigma)\subset X$ under the map $\trop$, see Figure~\ref{fig: trop-p2}. 

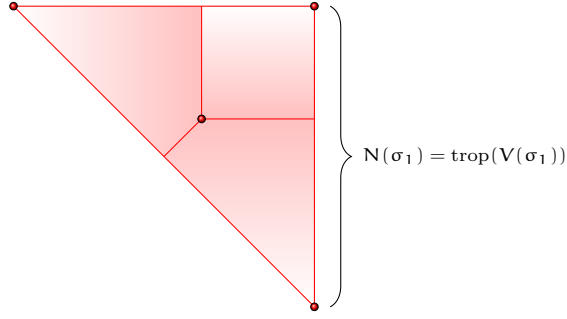
\begin{figure}
\begin{tikzpicture}

\fill[pink, path fading=north] (0.5,1.5)--(2,1.5)--(2,3)--(0.5,3) -- cycle;
\fill[pink, path fading=south] (0.5,1.5)--(2,1.5)--(2,-1)--(0,1) -- cycle;
\fill[pink, path fading=west] (0.5,1.5)--(0,1)--(-2,3)--(0.5,3) -- cycle;

\draw [red] (-2,3)--(2,3)--(2,-1)--(-2,3);

\draw[red] (0.5,3)--(0.5,1.5)--(2,1.5); \draw[red] (0.5,1.5)--(0,1);

\draw [ball color = red] (-2,3) circle (0.5mm);
\draw [ball color = red] (2,3) circle (0.5mm);
\draw [ball color = red] (2,-1) circle (0.5mm);
\draw [ball color = red] (0.5,1.5) circle (0.5mm);

\draw [decorate,decoration={brace,amplitude=8pt},xshift=-4pt,yshift=0pt] (2.35,3) -- (2.35,-1) node [black,midway,xshift=-0.6cm] {};

\node at (4,1) {\tiny $N(\sigma_1) = \trop(V(\sigma_1))$};
\end{tikzpicture}
\caption{The extended tropicalization of $\PP^2$ consists of $3$ extended cones. The vector space at infinity that is perpendicular to each ray $\sigma_i$ is the tropicalization of the locally closed stratum $V(\sigma_i)$. }
\label{fig: trop-p2}
\end{figure}

Let $S_\sigma$ be a toric monoid and consider the scheme $X = \spec(\CC\llbracket S_\sigma \rrbracket)$. The space $X^\beth$ consists of those valuations on $\CC\llbracket S_\sigma \rrbracket$ that are nonnegative on the monoid $S_\sigma$. As before, restriction of valuations to $S_\sigma$ determines a continuous map
\[
\trop: X^\beth\to \Hom(S_\sigma,\RR_{\geq 0}\sqcup \{\infty\}). 
\]
That is, tropicalization is narturally a map from $X^\beth$ to the canonical compactification $\overline \sigma$ of the dual cone $\sigma$ of the monoid $S_\sigma$. This tropicalization map can be generalized to toroidal embeddings over $\CC$. While the theory works in greater generality, we restrict ourselves to toroidal embeddings \textit{without self-intersection}~\cite[Section II.1]{KKMSD}. 

\begin{definition}
A \textbf{toroidal embedding} is a normal variety $X$ together with an open set $U\subset X$ with complemenet $D = X\setminus U$ such that, at each point $x\in X$, there is an affine toric variety $V(\sigma,x)$ and a point $t\in V(\sigma,x)$ together with an isomorphism of complete local rings
\[
\widehat{\mathscr O}_{X,x}\xrightarrow{\sim} \widehat{\mathscr O}_{V(\sigma,x),t},
\]
where the ideal of $D$ in $\mathscr O_{X,x}$ maps to the ideal of the toric boundary in $V(\sigma,x)$ under this isomorphism. 
\end{definition}

The definition is a systematization of the intuitive notion that $X$ is \textit{formally locally isomorphic to a toric variety}. A morphism $X\to Y$ of toroidal embeddings is said to be \textit{toroidal} if in local toric charts, it is a dominant equivariant morphism of toric varieties. A \textit{toroidal modification} is a toroidal morphism $X\to Y$ that is birational and given by a toric modification in local charts, i.e. by a subdivision of the corresponding fans. We refer the reader to~\cite[Section 1]{AK00} for details on toroidal morphisms.

Let $X$ be a toroidal embedding over $\CC$, with its trivial valuation. By work of Thuillier~\cite[Section 3.2]{Thu07}, the maps $\trop$ on each formal local chart glue to form a global tropicalization map
\[
\trop: X^\beth\to \overline \fS(X),
\]
where $\overline \fS(X)$ is the compactified cone complex associated to the fan of the toroidal embedding $X$, as defined in~\cite[Section II, p.71]{KKMSD}. This global tropicalization map is a deformation retraction. In~\cite[Section 6]{ACP}, Thuillier's deformation retraction is extended to toroidal Deligne--Mumford stacks. Given a toroidal Deligne--Mumford stack $\cX$ with coarse space $X$, the authors of loc. cit. produce a strong deformation from $X^{\beth}$ onto a generalized extended cone complex $\overline \fS(\cX)$. It follows from~\cite[Proposition 5.7]{U14b}, that the underlying topological space of the analytic stack $\cX^{\beth}$ coincides with $X^{\beth}$. The formation of the skeleton of a toroidal stack is functorial for toroidal morphisms.

Toroidal embeddings without self-intersection over $\CC$ are precisely the logarithmically regular varieties whose log structure is defined on the Zariski site, rather than the \'etale site. In~\cite[Theorems 1.1,1.2]{U13} this tropicalization map is generalized to a functorial tropicalization for fine and saturated \'etale logarithmic schemes over $\CC$.

\subsection{Logarithmic stable maps} In this section we review some basic notions of logarithmic stable maps. We freely use the Kato--Fontaine--Illusie theory of logarithmic geometry in this paper, and refer the reader to the survey~\cite{ACGHOSS}, in addition to K. Kato's seminal article~\cite[Sections 1-4]{Kat89} for background. See~\cite{ACMUW} for a survey on non-archimedean logarithmic geometry.

A \textit{pre-stable log curve} is a morphism $\pi:X\to S$ of fine and saturated logarithmic schemes such that $f$ is logarithmically smooth, flat, with reduced, connected fibers of pure dimension $1$. A \textit{marked pre-stable log curve} in addition comes with sections $s_i:\underline S\to \underline X$, that are distinct from the double points of the fibres. Moreover, we require that $X$ carries the divisorial log structure along the sections $s_i$. A \textit{marked stable log curve} is a marked pre-stable log curve whose underlying marked curve is stable in the usual sense.

\begin{center}\textit{
$(\star)$ For the rest of the paper, unless otherwise stated, we assume that the target variety $X$ is a projective toric variety associated to a fan $\Delta$.}
\end{center}

\begin{definition}
A \textbf{logarithmic stable map} over $S$ is a pre-stable marked log curve $(C\to S,\{s_i\})$, together with a map $f:C\to X$ of logarithmic schemes, such that the underlying map $\underline f: \underline C\to \underline X$ is an ordinary stable map. 
\end{definition}

At the marked sections $s_i$ of the family $C\to S$, the relative characteristic sheaf of $C$ over $S$ has stalk $\NN$. If the map $f$ sends $s_i$ to the stratum $V(\sigma)$ of $X$, the logarithmic structure gives a homomorphism $c_i: M_\sigma\to \NN$, which by dualizing, is an integral element of the cone $\sigma\subset N_\RR$. 

\begin{definition}
The collection of homomorphisms $\{c_i\}$ is referred to as the \textbf{contact order} of the stable map $f$, and is denoted $c$. A marked point $p_i$ has \textbf{trivial contact order} if the homomorphism $ c_i:M\to \NN$ is the zero map. 
\end{definition}

\begin{remark}{\bf (Interpretation of the contact order)}
Suppose the curve $C$ meets the dense torus, and the marked point $p_i$ maps to the relative interior of a toric boundary divisor $D$. The characteristic sheaf of a divisor $D$ is isomorphic to $\NN$, and can be understood as the multiplicative monoid generated by a function $g$ that cuts out $D$ locally in $X$. The contact order is the order of vanishing of the image of $C$ at $p_i$, i.e. the order of tangency between $D$ and $C$ at the marked point $p_i$. A point $p_i$ with trivial contact order is mapped to the dense torus of $X$.The logarithmic framework allows one to understand this order of tangency when the image of $C$ lies in $X$. 
\end{remark}

As a convention, we separate the $m$ sections with trivial contact orders from the $n$ sections with nontrivial contact orders, so our logarithmic stable maps will carry $(m+n)$ marked sections. 

\begin{definition}
Let $[f:C\to X]$ be a logarithmic stable map. The collection $\Gamma$ of the genus $g$ of $C$, number $m$ of marked points with trivial and $n$ with non-trivial contact orders, the contact orders of these points, and the curve class $\beta$ will be referred to as the \textbf{discrete data} of $f$.
\end{definition}

For toric targets the contact order uniquely determines the curve class $\beta$. Indeed, the operational Chow class of any curve on a complete toric variety is determined by the degrees of its intersections with the boundary divisors by~\cite[Theorem 2.1]{FS97}. 

\begin{theorem}[{Abramovich--Chen~\cite{AC11} and Gross--Siebert~\cite{GS13}}]
The category of stable logarithmic maps to $X$ forms a logarithmic Deligne--Mumford stack that is finite and representable over the Kontsevich space $\Mbar_{g,n+m}(\underline X,\beta)$. 
\end{theorem}

We point out a subtle issue at play. The moduli functor of stable logarithmic maps describes a stack over the category of fine and saturated logarithmic schemes, rather than schemes. In order to apply standard techniques from algebraic geometry, one wishes to understand the functor parametrizing logarithmic stable maps over test schemes $\underline S$ without any logarithmic structure, i.e. as an algebraic stack with logarithmic structure, rather than stack over logarithmic schemes. This problem is solved by the concept of \textit{minimality}. Given a map $\underline S\to \Lsm_\Gamma(X)$, there is a \textit{minimal} logarithmic structure that one may place on $\underline S$, such that the map $\underline S\to \Lsm_\Gamma(X)$ parametrizes minimal logarithmic stable maps over $S$. In other words, given any logarithmic stable map over a logarithmic base $(S,\mathpzc M_S)$, we wish to find a minimal structure $(S,\mathpzc M_S^{\min})$ over which the given map factors. See~\cite[Section 2]{AC11} for a discussion in the context of logarithmic stable maps, and~\cite{Gi12} for a general categorical discussion. We warn the reader that some authors, including Gross and Siebert~\cite[Section 1.5]{GS13} and Kim~\cite{Kim08} refer to minimality as the \textit{basicness} condition.

For us, the following existence\footnote{Strictly speaking, Abramovich and Chen~\cite{AC11} show this result when $X$ admits a generalized Deligne--Faltings structure, so one must first check this hypothesis is satisfied for toric varieties. This is done in~\cite[Proposition A.4]{CS12}} result will suffice~\cite{AC11,GS13}. There is a proper algebraic stack $\Lsm_\Gamma(X)$ which represents the moduli functor which, for a test scheme $\underline S$, returns the groupoid of minimal logarithmic stable maps over $\underline S$ with discrete data $\Gamma$. As with standard Gromov--Witten theory, the stack $\Lsm_\Gamma(X)$ admits \textit{evaluation morphisms} to the strata of $X$. In this paper we will be concerned only with evaluations at the \textit{ordinary} marked points
\[
ev_i: \Lsm_\Gamma(X)\to X,
\]
i.e. those marked points with trivial contact orders with the toric boundary.

\section{The skeleton of $\Lsm$}\label{lsm-skeleton}

\subsection{Tropical stable maps} Given a finite tree $\underline G$, we refer to the non-leaf edges as \textit{internal edges}, and to the $1$-valent vertices adjacent to the leaves as \textit{infinite points}.

An \textit{abstract $n$-marked rational tropical curve} is a tree $\underline G$ with $n$ marked leaf edges $e_{p_1},\ldots, e_{p_n}$ and a length function $\ell: E(G)\to \RR_{\geq 0}\sqcup \{\infty\}$. We require that $\ell(e_{p_i}) = \infty$ for all leaves $e_{p_i}$. 

This data produces a topological space $G$ with a singular metric. To each edge $e$ of finite length $\ell(e)$ associate the metric space $[0,\ell(e)]$. To $e_{p_i}$ associate the extended interval $[0,\infty]$. If $e$ is an internal edge adjacent to vertices $v_1$ and $v_2$, such that $\ell(e) = \infty$, then associate to $e$ the space $[0,\infty]\sqcup_\infty [0,\infty]$. These extended intervals associated to edges glue along the adjacencies prescribed by $\underline G$, so have determined a topological space $G$ with a singular metric induced by the length. When $(\underline G, \ell)$ are clear from context, we will refer to $G$ itself as an \textit{abstract $n$-marked rational tropical curve}, see Figure~\ref{fig: tropical-curve}.

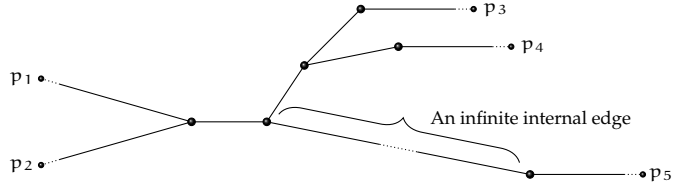
\begin{figure}[h!]
\begin{tikzpicture}
\draw (1,0)--(2.5,-0.3); \draw[densely dotted] (2.5,-0.3)--(3,-.4); \draw (3,-.4)--(4.5,-.7);

\draw (0,0)--(1,0)--(1.5,0.75)--(2.75,1);
\draw (1.5,0.75)--(2.25,1.5);

\draw (0,0)--(-1.75,0.5); \draw[densely dotted] (-1.75,0.5)--(-2,0.575);
\draw (0,0)--(-1.75,-0.5); \draw[densely dotted] (-1.75,-0.5)--(-2,-0.575);
\draw (2.75,1)--(4,1); \draw[densely dotted] (4,1)--(4.25,1);
\draw (2.25,1.5)--(3.5,1.5); \draw[densely dotted] (3.5,1.5)--(3.75,1.5);
\draw (4.5,-0.7)--(5.75,-0.7); \draw[densely dotted] (5.75,-0.7)--(6,-0.7);

\draw [ball color = black] (0,0) circle (0.5mm);
\draw [ball color = black] (1,0) circle (0.5mm); 
\draw [ball color = black] (1.5,0.75) circle (0.5mm);
\draw [ball color = black] (2.75,1) circle (0.5mm);
\draw [ball color = black] (2.25,1.5) circle (0.5mm);
\draw [ball color=black] (4.5,-.7) circle (0.5mm);

\draw [ball color=black] (-2,0.575) circle (0.35mm);
\draw [ball color=black] (-2,-0.575) circle (0.35mm);
\draw [ball color=black] (4.25,1) circle (0.35mm);
\draw [ball color=black] (3.75,1.5) circle (0.35mm);
\draw [ball color=black] (6,-0.7) circle (0.35mm);

\node at (-2.25,0.575) {\textnormal{\tiny $p_1$}};
\node at (-2.25,-0.575) {\textnormal{\tiny $p_2$}};
\node at (4.55,1) {\textnormal{\tiny $p_4$}};
\node at (4.05,1.5) {\textnormal{\tiny $p_3$}};
\node at (6.25,-0.7) {\textnormal{\tiny $p_5$}};

\draw [decorate,decoration={brace,amplitude=8pt},xshift=-4pt,yshift=0pt] (1.3,0.1) -- (4.5,-0.55) node [black,midway,xshift=1.75cm,yshift=0.25cm] {\textnormal{\tiny An infinite internal edge}};
\end{tikzpicture}
\caption{An abstract nodal $5$-marked tropical tropical curve. This curve has a single internal edge of infinite length.}
\label{fig: tropical-curve}
\end{figure}

If all non-leaf edges have finite length, we say that $G$ is \textit{smooth} and otherwise $G$ is called \textit{nodal}. In this paper we will not consider tropical curves of higher genus. 

\begin{remark}
Our terminology of smooth and nodal is motivated by the theory of Berkovich curves. The smooth tropical curves above are precisely those that arise as skeletons of marked semistable models for $\PP^1$ over a rank-$1$ valuation ring. Nodal tropical curves appear as skeletons of marked semistable models whose generic fiber is nodal. See for instance~\cite[Section 5]{BPR}. 
\end{remark}

\begin{definition}~\label{def: trop-stable-map}
A \textbf{tropical stable map from a smooth rational curve to $\overline \Delta$} is a genus $0$ abstract smooth marked tropical curve $G$, together with a continuous proper map
\[
f:G\to \overline \Delta,
\]
such that
\begin{enumerate}[(1)]
\item The set $f^{-1}(\overline \Delta\setminus \Delta)$ is a union of infinite points of $G$.
\item Each edge $e$ of $G$ is mapped to a single extended cone $\overline\sigma_e$ of $\overline \Delta$.
\item For every edge $e$ in $G$, $f(e)$ has rational slope $u_e\in N$. Moreover the map \[
f|_e: e\to f(e) 
\]
is linear of integer slope $w_e$, defined with respect to the primitive integral vector in the direction of $f(e)$. We refer to the absolute value of this slope as the \textbf{expansion factor} of $f$ along $e$. The \textbf{contact order} of $f$ along $e$ is the quantity $c_e = w_eu_e$.
\item The polyhedral complex $f(G)\cap N_\RR$ is a balanced weighted polyhedral complex, where an edge $e$ is given weight equal to the expansion factor of $f$ at $e$.
\item[(Stability)] If $v\in f(G)$ is a divalent vertex, then either (1) $f^{-1}(v)$ consists of a vertex of $G$ of valence at least $3$, or (2) in a neighborhood of $v$ in $f(G)$, $v$ is the unique intersection of $f(G)$ with an extended cone $\overline \sigma\in \overline \Delta$. 
\end{enumerate}
Let $e$ be a marked edge of $G$, such that the infinite point of $e$ is mapped to the locally closed stratum $N(\sigma_e)$. The edge $f(e)$ is parallel to a rational ray in $\sigma$ with primitive generator $u_e$. Define the \textbf{contact order} of $f$ along $e$ to be the quantity $c_e = w_e u_e\in N$.
\end{definition}

\begin{figure}[h!]
\begin{tikzpicture}
\fill[pink, path fading=north] (0,0)--(0,2)--(2,2)--(2,0) -- cycle;
\fill[pink, path fading=south] (0,0)--(2,0)--(2,-1.414)--(-1.414,-1.414) -- cycle;
\fill[pink, path fading=west] (0,0)--(0,2)--(-1.414,2)--(-1.414,-1.414) -- cycle;

\begin{scope}[shift = {(5,0)}]
\fill[pink, path fading=north] (0,0)--(0,2)--(2,2)--(2,0) -- cycle;
\fill[pink, path fading=south] (0,0)--(2,0)--(2,-1.414)--(-1.414,-1.414) -- cycle;
\fill[pink, path fading=west] (0,0)--(0,2)--(-1.414,2)--(-1.414,-1.414) -- cycle;
\end{scope}

\draw[->,red] (0,0)--(0,2);
\draw[->,red] (0,0)--(2,0);
\draw[->,red] (0,0)--(-1.414,-1.414);

\draw (2.5,1.5)--(1,1.5)--(1,3);
\draw (1,1.5)--(-0.5,0);

\draw [ball color=black] (1,1.5) circle (0.5mm);
\draw [ball color=black] (0,0.5) circle (0.5mm);

\draw[->,red] (5,0)--(5,2);
\draw[->,red] (5,0)--(7,0);
\draw[->,red] (5,0)--(3.686,-1.414);

\draw (7.5,1.5)--(6,1.5)--(6,3);
\draw (6,1.5)--(4.5,0);

\draw [ball color=black] (6,1.5) circle (0.5mm);
\draw [ball color=black] (5,0.5) circle (0.5mm);
\draw [ball color=black] (6,2.5) circle (0.5mm);
\node at (4.75,2.5) {\tiny \begin{tabular}{c} An unstable \\ divalent vertex \end{tabular}};

\draw[dashed] (6,2.5)--(9,3.5);
\draw[dashed] (6,2.5)--(9,1.5);

\draw [color = black] (9,2.5) circle (0.8);
\draw (8.2,2.5) to [bend right] (9.8,2.5);
\draw [dotted] (8.2,2.5) to [bend left] (9.8,2.5);
\draw [ball color=black] (9,3.3) circle (0.5mm);
\draw [ball color=black] (9,1.7) circle (0.5mm);

\draw [ball color=red] (0,0) circle (0.5mm);
\draw [ball color=red] (5,0) circle (0.5mm);
\end{tikzpicture}
\caption{A stable (left) and unstable (right) $3$-marked tropical map to $\Delta_{\PP^2}$. The unstable vertex corresponds to a contracted twice marked $\PP^1$.}
\label{fig: stable-unstable}
\end{figure}
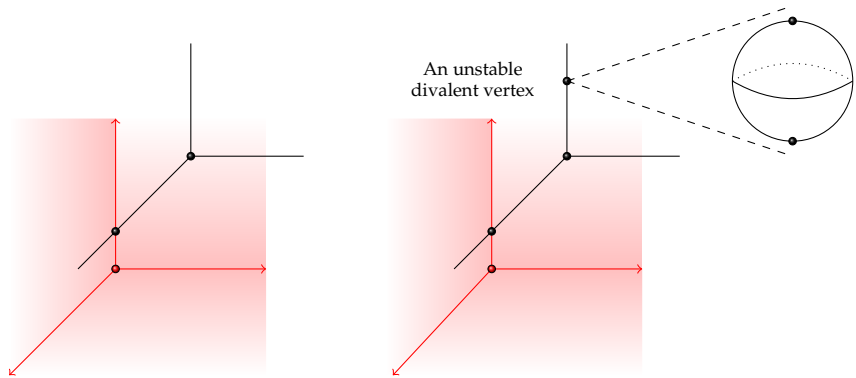

\begin{remark}
A few remarks are in order concerning this definition. We note that (2) above has usually been ignored in previous works~\cite{GKM07,Mi03}. It is equivalent to the statement that $f$ is a map of polyhedral complexes, and can always be achieved after a unique minimal subdivision of $G$. These subdivisions will affect the cone complex structure of the moduli space of maps which is important for our purposes. The \textit{stability condition} may be unfamiliar to the reader. The first part of the stability condition is classical: a contracted rational component must have $3$ special points. The second part of the stability condition ensures that no twice marked component of the special fiber of a degenerating map is contracted. In Theorem~\ref{thm: pointwise-trop} we will prove that a family of logarithmic maps has stable special fiber if and only if its tropicalization is stable. See Figure~\ref{fig: stable-unstable}.
\end{remark}

\begin{definition}
Let $f$ be a tropical stable map. \medskip

\noindent
The \textbf{discrete data} $\Gamma$ of $f$ is the number $n+m$ of marked edges of $G$, together with the contact orders $c_e\in N$ of each marked edge $e$. \medskip

\noindent
The \textbf{combinatorial type of $f$} is the following data.
\begin{enumerate}
\item The underlying combinatorial graph $\underline G$ of $G$, including the labeling of the $m+n$ marked edges.
\item For each vertex $v$, the cone $\sigma_v$ containing $v$.
\item For each edge $e$, the contact order $c_e$ of $f$ at $e$.
\end{enumerate}
\end{definition}


\subsection{Constructing the moduli space of tropical maps}\label{sec: trop-moduli} Let $\Theta$ be a combinatorial type for a tropical stable map. Suppose $\Theta$ has $n$ marked infinite edges with non-zero contact order and $m$ infinite edges with zero contact order.  Denote by $\underline G_\Theta$ the underlying combinatorial source graph, by $u_e\in N$ the primitive integral vector parallel to which the edge $e\in G$ maps, and by $\sigma_v$ the cone containing the vertex $v$.

Given a tropical curve $G$, let $G^{\mathrm{stab}}$ be the \textit{stabilization} of $G$ defined as follows. Consider a two valent vertex $v$ in $G$ with incident edges $e_1$ and $e_2$ and adjacent vertices $u_1$ and $u_2$. We may ``straighten'' $G$ by deleting both $e_1$ and $e_2$, connecting $u_1$ and $u_2$ by a single new edge of length $\ell(e_1)+\ell(e_2)$. This may be visualized by simply erasing the $2$-valent vertex. The resulting tropical curve is the stabilization of $G$. The \textit{overvalence} $\mathrm{ov}(G)$ of $G$ is the sum over finite vertices of $G^{\mathrm{stab}}$
\[
\mathrm{ov}(G)=\sum_{v\in V(G^{\mathrm{stab}})} \mathrm{deg}(v)-3.
\]

\begin{proposition}
The collection of all tropical stable maps from smooth curves with fixed combinatorial type $\Theta$ is a rational polyhedral cone $\sigma_\Theta$ of dimension $\dim X-3+m+n-\mathrm{ov}(G)$. 
\end{proposition}

\begin{proof}
Let $E$ be the number of bounded edges in the source graph $G$ of $\Theta$. We will describe $\sigma_\Theta$ as a subcone of 
\[
\tau_\Theta = \prod_{v_i: \mathrm{vertex}} \sigma_{v_i} \times \RR_{\geq 0}^E. 
\] 
A point of $\tau_\Theta$ prescribes a positive length $\ell(e)$ to each $e$ of $G$, and a position $f(v)\in \sigma_v$ of each vertex $v$ of $G$. The cone $\sigma_\Theta$ is precisely the subcone where these assignments describe a tropical stable map. This is ensured by the following condition: given an edge $e$ with endpoints $v_1$ and $v_2$, we require
\[
v_2-v_1 = \ell(e)\cdot c_e,
\]
where $c_e$ is the contact order prescribed by the combinatorial $\Theta$.
Ranging over all edges, this cuts out a closed subcone $\sigma_\Theta$ of $\tau_\Theta$. This cone is easily seen to be rational and polyhedral, and inherits an integral structure from $\tau_\Theta$.

For the statement about the dimension of $\sigma_\Theta$, observe that by straighting $2$-valent vertices, each tropical stable map gives rise to a unique map $G^{\mathrm{stab}}\to \overline \Delta$, but edges may no longer map to unique cones of $\overline \Delta$. The dimension of $\sigma_\Theta$ is the same as the dimension of the cone of such maps, and we conclude the result from~\cite[Section 1]{NS06}.
\end{proof}

The faces of $\sigma_\Theta$ are cones associated to combinatorial types.

\begin{proposition}\label{prop: types}
A moduli cone $\sigma_{\Theta'}$ is a face of $\sigma_{\Theta}$ if and only if 
\begin{enumerate}[(F1)]
\item The source type $G'$ and $\Theta'$ is obtained from the source graph $G$ of $\Theta$ by a (possibly empty) collection of edge contractions $\alpha: G\to G'$.
\item Given any vertex $v'\in G'$ and a vertex $v$ such that $\alpha(v) = v'$, then the cone $\sigma_{v'}$ is a face of $\sigma_v$.
\end{enumerate}
In particular, each face of $\sigma_\Theta$ parametrizes tropical stable maps the cone associated to a combinatorial type with discrete data $\Gamma$.
\end{proposition}

\begin{proof}
The faces of $\sigma_\Theta$ correspond precisely to those faces of 
\[
\tau_\Theta = \prod_{v_i: \mathrm{vertex}} \sigma_{v_i} \times \RR_{\geq 0}^E,
\]
that intersect $\sigma_\Theta$. Since the coordinates on $\tau_\Theta$ record the position of a vertex $v$ and lengths of an edge $e$, the result follows. 
\end{proof}

Denote by $\Tsm_\Gamma^\circ(\Delta)$ the cone complex formed by gluing the cones $\sigma_\Theta$ along faces, as dictated by Proposition~\ref{prop: types} above.

\begin{definition}
The canonical compactification of the cone complex $\Tsm_\Gamma^\circ(\Delta)$ is the \textbf{moduli space of tropical stable maps to $\overline \Delta$ with discrete data $\Gamma$} and is denoted $\Tsm_\Gamma(\Delta)$.
\end{definition}

\begin{remark}
The points $p\in \Tsm_\Gamma(\Delta)\setminus \Tsm_\Gamma^\circ(\Delta)$ naturally parametrize tropical stable maps from tropical nodal curves to $\overline \Delta$, but we will have no need to work directly with these objects.
\end{remark}

The extended cone complex $\Tsm_\Gamma(\Delta)$ admits natural \textit{tropical evaluation morphisms}, 
\[
ev^\trop_i: \Tsm_\Gamma(\Delta)\to \overline \Delta,
\]
sending a map $[f:C\to \overline \Delta]$ to the image of the infinite point of the marked edge $e_{p_i}$. Observe that if $e_{p_i}$ is a marked edge with trivial contact order, the expansion factor along $e_{p_i}$ is $0$, so this edge is contracted. As a result, for such an edge, $ev_i^\trop$ restricts to a map of cone complexes
\[
ev_i^\trop: \Tsm_\Gamma^\circ(\Delta)\to \Delta.
\]
These evaluation morphisms for marked edges with contact order zero are precisely the evaluations defined by Gathmann, Kerber, and Markwig~\cite[Definition 4.2]{GKM07}.

\subsection{The skeleton of $\Lsm$} Fix discrete data $\Gamma = (n,m,\beta,c)$, of $n$ marked points with contact orders given by $c$, $m$ additional marked points with trivial contact order, and the curve class $\beta$, determined by $c$. Let $\Lsm_\Gamma(X)$ denote the moduli space of minimal logarithmic stable maps to $X$ with discrete data $\Gamma$, and by $\Lsm^\circ_\Gamma(X)$ the locus where the logarithmic structure is trivial.

\begin{proposition}\label{prop: logsmooth}
The moduli space $\Lsm_\Gamma(X)$ is a logarithmically smooth Deligne--Mumford stack over $\spec(\CC)$ of dimension $\dim X-3+m+n$.
\end{proposition}

\begin{proof}
There is a natural forgetful morphism
\[
\pi: \Lsm_\Gamma(X)\to \fM_{0,n+m},
\]
to the Artin stack of pre-stable marked curves, sending a map $\xi = [f: C\to X]$ to its marked source curve. We view $\fM_{0,n+m}$ as a logarithmically smooth stack as in~\cite{Ols03}, with the natural divisorial log structure in the smooth topology. To prove that $\Lsm_\Gamma(X)$ is logarithmically smooth at $\xi$, it suffices to show that logarithmic deformations of the map $f$ keeping the marked curve $C$ fixed, are log unobstructed. By~\cite[Section 5]{GS13}, relative logarithmic obstructions lie in the group $H^1(C,f^\star T_X^{\mathrm{log}})$. The logarithmic tangent bundle of the toric variety $X$ is $\mathscr O_X^{\dim X}$, and $C$ has arithmetic genus $0$, so this group vanishes. Thus $\Lsm_\Gamma(X)$ is logarithmically smooth. 

To see that the dimension is as claimed, observe the relative dimension of $\pi$ is the dimension of the space of deformations of the map $f$, where the source curve is fixed. That is, $\mathrm{relDim}(\pi) = \dim H^0(C,f^\star T_X^{\mathrm{log}}) = \dim X$. The stack $\fM_{0,n+m}$ has dimension $n+m-3$, so the claim follows.
\end{proof}

The contact order $c$ determines a fan $\Sigma(c)$ in $N_\RR$ as follows. For each marked point $p_i$ of nontrivial contact order, the contact order $c(p_i)$ is equivalent to the data of a point $v(p_i)\in N$. Define $\Sigma(c)$ to be the one dimensional fan whose cones are the rays $\langle v(p_i)\rangle$. 

\begin{definition}
The contact order $c$ is said to be \textbf{torically transverse} if $\Sigma(c)$ is supported on the $1$-skeleton of $\Delta$.
\end{definition}

We now describe the open locus of the moduli space of maps. Related statements appear in~\cite{GM07,Gro14,Kap93,Tev07}. 

\begin{proposition}\label{prop: lsm-m0n}
Suppose $c\in \Gamma$ is a torically transverse contact order and $m\geq 1$. Then there is a natural isomorphism
\[
\Lsm_\Gamma^\circ(X)\cong M_{0,n+m}\times T.
\]
\end{proposition}

\begin{proof}
When the target is a point there is nothing to prove, so assume that $\dim X\geq 1$.  On the interior of the moduli space $\Lsm_\Gamma(X)$ the logarithmic structure is trivial. Thus, for any logarithmic map $[f:\PP^1\to X]$, the marked points map to locally closed torus orbits of codimension at most $1$. Thus, we may replace $X$ with a toric resolution of singularities without changing $\Lsm_\Gamma^\circ(X)$. Since $X$ is now smooth, we work with the homogeneous coordinate ring of $X$ in the sense of Cox~\cite{Cox95}.

Consider a map $[f:\PP^1\to X]\in \Lsm^\circ_\Gamma(X)$ and fix homogeneous coordinates $(x\!:\!y)$ on the source $\PP^1$. Since $m\geq 1$ and $X$ is proper of dimension at least $1$, it follows that $n+m\geq 3$, and we may fix the first $3$ points at $0,1,\infty$. Consider a prime toric boundary divisor $D_\rho$ associated to a ray $\rho\in \Delta$, and label the marked points of $\PP^1$ mapping to $D_\rho$ as $f^{-1}D_\rho = \{p^{(1)}_{\rho},\ldots,p^{(k)}_\rho\}$. In homogeneous coordinates on $\PP^1$ we have $p_\rho^{(j)} = (a_\rho^j\!:\!b_\rho^j)$. By applying~\cite[Theorem~2.1]{Cox95} any map 
\[
f: \PP^1\to X
\]
with discrete data $\Gamma$ is given by a collection of homogeneous polynomials $(f_\rho)_{\rho\in \Delta^{(1)}}$. Since we have chosen coordinates for this preimage in the source $\PP^1$, we see that we can write $f_\rho$ explicitly as
\[
f_\rho = \lambda_\rho \prod_{j=1}^{k} (b_\rho^{(j)}x-a_\rho^{(j)}y)^{c(p_\rho^{(j)})}. 
\]
where $\lambda_\rho\in \CC^\times$. Thus, $f$ is fully determined by the image of the distinguished marking $p_1$ mapping to $T$. It is now straightforward that to see that we may choose the desired isomorphism to be $\pi\times ev_1$, where $\pi$ sends $[f]$ to its marked source curve. 
\end{proof}

We will need the following variation of the above proposition.

\begin{lemma}\label{lem: fixed-boundary}
Let $\Lsm_\Gamma^\circ(X)$ be as above. Let $p$ be marked point with nontrivial contact order mapping to a locally closed stratum $V(\sigma_p)$. The locus $U\subset \Lsm_\Gamma$ parametrizing maps to $X$ such that $p$ is mapped to any fixed point of $V(\sigma_p)$ is irreducible.
\end{lemma}

\begin{proof}
The locus $U$ can be parametrized by those maps $f$ where $f(p)$ is fixed. Using the explicit description of $\Lsm^\circ_\Gamma(X)$ above, it is straightforward to check that this locus is irreducible.
\end{proof}

The proposition above suggests that the most naive compactification of the space $\Lsm^\circ_\Gamma(X)$ is simply $\overline M_{0,n+m}\times X$. However, this space does not admit natural evaluation morphisms to $X$ and does not remember data about the contact order. Thus, it contains insufficient information for Gromov--Witten theory. Our proof of Theorem~\ref{thm: trop-comp} will show that the logarithmic stable map compactification differs from this naive compactification only by an explicit toroidal modification. When $X$ is Calabi-Yau (as opposed to log Calabi-Yau as in our setting), the space $\overline M_{0,n}\times X$ as appeared in earlier work of Morrison and Plesser in mirror symmetry~\cite{MP95}.

\begin{proposition}\label{prop: irred}
The moduli space $\Lsm_\Gamma(X)$ is irreducible.
\end{proposition}

\begin{proof}
Since $\Lsm_\Gamma(X)$ is logarithmically smooth, it has a dense open set where the logarithmic structure is trivial, so it suffices to prove that $\Lsm_\Gamma^\circ(X)$ is irreducible. Observe that we may modify $\Gamma$ to a new discrete datum $\widetilde \Gamma$, by adding a marked point $p_1$ with trivial contact order. There is a surjective morphism 
\[
\Lsm_{\widetilde \Gamma}^\circ(X)\to \Lsm^\circ_\Gamma(X),
\] 
forgetting the marked point $p_1$ so the irreducibility of $\Lsm_{\widetilde \Gamma}^\circ(X)$ will imply the irreducibility of $\Lsm^\circ_\Gamma(X)$. Thus, we may and do assume that $\Gamma$ has at least one marked point $p_1$ with trivial contact. 

We now reduce to proving the proposition when the contact order is transverse to the toric boundary. The contact order $c$ determines a fan $\Sigma_c$ in $N_\RR$ as follows. For each marked point $v$, the contact order $c$ is equivalent to a point $v_p\in N$. Ranging over the marked points $p$ with nontrivial contact order, define $\Sigma_c$ to be the one-dimensional fan whose rays are $\langle v_p\rangle$. Choose a fan $ \Delta_c$ refining $\Delta$ such that $\Sigma_c$ is a subfan of $\Delta_c$, and let $X_c$ be the associated toric variety. The map $X_c\to X$ is a toric modification and by~\cite[Theorem B.6]{AMW12} there is an associated map $\Lsm_\Gamma(X_c)\to \Lsm_\Gamma(X)$. By applying~\cite[Proposition 5.3.1]{AW}, we conclude that 
\[
\Lsm_\Gamma(X_c)\to \Lsm_\Gamma(X) 
\]
is a toroidal modification, and in particular restricts to the identity on the locus $\Lsm^\circ_\Gamma(X_c)$. The result now follows from irreducibility of $M_{0,n}$ and Proposition~\ref{prop: lsm-m0n}.
\end{proof}

\begin{remark}
The irreducibility of the space of logarithmic maps to log homogeneous varieties has been a topic of recent interest, see~\cite[Proposition 2.4]{CS12} and~\cite{CZ14}.
\end{remark}

\begin{proposition}
There is a continuous idempotent self-map 
\[
\bm p: \Lsm^{\an}_\Gamma(X)\to \Lsm^{\an}_\Gamma(X),
\] 
giving a deformation retraction of $\Lsm^{\an}_\Gamma(X)$ onto a connected extended cone complex $\overline \fS$. 
\end{proposition}

\begin{proof}
By~\cite[Theorem 3.5]{Kat89}, the logarithmic smoothness of $\Lsm_\Gamma(X)$ over $\spec(\CC)$ is equivalent to $\Lsm_\Gamma(X)$ being toroidal. We apply Thuillier's deformation retraction  toroidal Deligne--Mumford stacks by Abramovich, Caporaso, and Payne in~\cite[Section 6]{ACP}. Since the projection $\Lsm^{\an}_\Gamma(X)\to \overline \fS$ is a homotopy equivalence, the connectivity of $\overline \fS$ is equivalent to the connectivity $\Lsm^{\an}_\Gamma(X)$. By non-archimedean GAGA~\cite[Theorem 3.5.3(iii)]{Ber90} this is equivalent to the connectivity of $\Lsm_\Gamma(X)$ which follows immediately from Proposition~\ref{prop: irred} above. 
\end{proof}

\subsection{Pointwise tropicalization}\label{sec: pointwise-trop} A point of $\Lsm^{\circ,\an}_\Gamma(X)$ is represented by a map
\[
\spec(K)\to \Lsm^\circ_\Gamma(X).
\]
We have a chosen compactification $\Lsm_\Gamma(X)$ of $\Lsm^\circ_\Gamma(X)$, so after replacing $K$ and $R$ by a finite ramified extension we obtain a family of logarithmic stable maps over a valuation ring $R$. Such a family produces a tropical stable map to $\overline \Delta$ as follows. Pull back the universal curve, map, and minimal logarithmic structure to obtain a diagram
\[
\begin{tikzcd}
C\arrow{d}\arrow{r} & \mathscr C \arrow{d}\arrow{r}{f} & X \\
\spec(K)\arrow{r} & \spec(R) & 
\end{tikzcd}
\]
where $\mathscr C$ is a semistable model with generic fiber $C$, and $f$ is a minimal logarithmic stable map. Let $G_{\mathscr C}$ denote the marked dual graph of the special fiber of $\mathscr C$, metrized as follows. Let $e$ be an edge associated to a node $q$. Define the length of $e$ by $\ell(e) = \nu(\pi)$, where $\pi\in R$ is a deformation parameter for the node $q$. It is straightforward to check that this assignment $e\mapsto \ell(e)$ is independent of all choices, see~\cite[Lemma 2.2.4]{Viv12} for a proof. This gives rise to an abstract tropical curve $G_{\mathscr C}$. The metric space $G_\mathscr C$ is naturally a subspace of the analytic space $C^{\an}$.

\begin{proposition}
There is a continuous map $\trop: C^{\an}\to G_{\mathscr C}$, together with a section $s: G_{\mathscr C}\to C^{\an}$, such that the pair of maps
\[
\begin{tikzcd}
C^{\an}\arrow[bend left]{r}{\trop} & G_{\mathscr C} \arrow[bend left]{l}{s} 
\end{tikzcd}
\] 
realize $G_{\mathscr C}$ as a strong deformation retract of $C^{\an}$.
\end{proposition}

\begin{proof}
This is proved by Berkovich in~\cite[Chapter 4]{Ber90}. For a proof using the theory of semistable vertex sets, see~\cite[Section 5]{BPR}.
\end{proof}

By functoriality of analytification, there is a map $f^{\an}:C^{\an}\to X^{\an}$, so we obtain a continuous map $f_\trop$ defined as the composite
\[ 
\begin{tikzcd}
G_{\mathscr C} \arrow[swap]{r}{s} \arrow[bend left]{rrr}{f_\trop} & C^{\an} \arrow[swap]{r}{f^{\an}} & X^{\an} \arrow[swap]{r}{\trop} & \overline \Delta.
\end{tikzcd}
\]
We now come to the main result of this subsection, that $f_{\trop}$ is a tropical stable map. 

\begin{theorem}\label{thm: pointwise-trop}
Let $f:\mathscr C\to X$ be a family of logarithmic stable maps with discrete data $\Gamma$ over $\spec(R)$. Assume that the moduli map sends $\spec(K)$ to $\Lsm^\circ_\Gamma(X)$. Then the map $f_{\trop}$ associated to $\mathscr C\to X$ is a tropical stable map from a smooth tropical curve, with discrete data $\Gamma$. 
\end{theorem}

\begin{proof} We check the conditions of Definition~\ref{def: trop-stable-map}.
\
\medskip

\noindent
{\bf The source graph}.  By hypothesis, the generic point of $\spec(R)$ maps to the locus in $\Lsm_\Gamma(X)$ where the logarithmic structure is trivial so the generic fiber $C$ of $\mathscr C$ must be smooth. The length of an internal edge $e$ is given by the valuation of the smoothing parameter $\pi_e$ at the corresponding node. If $\nu(\pi_e) = \infty$, then $\pi_e = 0$, which contradicts the fact that $C$ is smooth. Thus, the lengths of all internal edges are finite. Since $\mathscr C$ is a flat family of curves, the arithmetic genus of the special fiber is $0$, so the dual graph must be a tree. We conclude that $G_{\mathscr C}$ is a smooth $(n+m)$-marked rational tropical curve. \\

\noindent
{\bf The infinite points}. As above, since the generic point of $\spec(R)$ maps to the locus $\Lsm_\Gamma^\circ(X)$, the logarithmic structure on $C$ is trivial away from the markings. This implies that $f$ must map $C$ to the dense torus $T\subset X$, except possibly at the marked points. Thus $C^{\an}$ minus finitely many points of $C(K)\subset C^{\an}$ is mapped to $T^{\an}$ under $f^{\an}$. We conclude that  $f_{\trop}^{-1}(\overline \Delta\setminus \Delta)$ consists of infinite points of $G_{\mathscr C}$. \\

\noindent
{\bf The edges of $G_{\mathscr C}$.} Let $p$ be a marked section of $\mathscr C$. The skeleton $G_{\mathscr C}$ contains a marked infinite edge $e_p$ corresponding to $p$. Assume that the reduction $\overline p$ of $p$ to the special fiber of $\mathscr C$ maps to the stratum $V(\sigma)$ in $X$. To compute the image of $e_p$ in $\overline \Delta$, observe that the contact order gives rise to a homomorphism
\[
c_p: \overline{\mathpzc M}_{X,f(\overline p)}\to \NN.
\]
Here the target monoid $\NN$ is the stalk of the relative characteristic of the source family $\mathscr C\to \spec(R)$. Dualizing, $c_p$ is a point of $\sigma$ and may write $c_p = w_pu_p$, where $u_p$ is primitive and $w_p\in \NN$.  We claim that the edge $e_p$ is mapped onto its image with expansion factor $w_p$. Let $U$ be an \'etale local chart for $\mathscr C$ near $\overline p$, where the section $p$ is cut out by a parameter $x_p$. The edge $e_p\subset G_{\mathscr C}\subset C^{\an}$ can then be parametrized as the set of monomial valuations
\begin{eqnarray*}
\val_r: R[U]&\to& \RR\sqcup \{\infty\} \\
\sum a_i x_p^i &\mapsto& \min_i\{ \nu(a_i)+r\cdot i\},
\end{eqnarray*}
for $r\in [0,\infty]$. Note that the valuation is independent of the choice of chart by~\cite[Lemma 2.2.4]{Viv12}. The data of the contact order $c_p$ is equivalent to the statement that the monomial $x_p$ is the pullback of the character $\chi^{c_p}\in M$. Composing with the projection $X^{\an}\to \overline \Delta$ amounts to taking valuations, so we see that $e_p$ maps to the real ray generated by $u_p$, with expansion factor equal $w_p$.  

Consider an edge $e_q\in G_{\mathscr C}$ of length $\ell_q$, corresponding to a node $q$ of the special fiber of $\mathscr C$. Let $V$ be an \'etale neighborhood of $q$ in $\mathscr C$, where functions $f$ and $g$ cut out the divisors meeting at $q$. The edge $e_q$ can be parametrized as the set of monomial valuations
\begin{eqnarray*}
\val_r: R[V]&\to& \RR\sqcup \{\infty\} \\
\sum c_{ij} f^i g^j &\mapsto& \min_i\{ \nu(c_{ij})+ r\cdot i+(\ell_q-r)\cdot j \},
\end{eqnarray*}
for $r\in [0,\ell_q]$. A similar argument shows that the edges of $G_{\mathscr C}$ are mapped onto their images with integer expansion factor. 

\noindent
{\bf Edges map to single cones.} We now check that every edge is mapped to a unique extended cone of $\overline \Delta$. Let $e$ be an edge of $G$ incident to vertices $v_1$ and $v_2$. Assume that $f_{\trop}$ maps $v_1$ and $v_2$ to cones $\sigma_1$ and $\sigma_2$ respectively. It follows that the generic point of the associated component $C_{v_1}$ of the special fiber of $\mathscr C$ maps to the locally closed stratum $V(\sigma_1)$. This implies that the node $q_e$ corresponding to $e$ is mapped to the torus orbit closure $\overline{V(\sigma_1)}$. Repeating this argument for $C_{v_2}$ shows that the node $q_e$ is mapped to the orbit closure $\overline{V(\sigma_2)}$, and that the closures of the orbits $V(\sigma_1)$ and $V(\sigma_2)$ intersect in a locally closed orbit $V(\sigma_e)$ containing $q_e$. Thus, $\sigma_e$ must contain $\sigma_1$ and $\sigma_2$ as faces. This implies that the image of $e$ lies in a single extended cone, namely, $\sigma_e$. 

\noindent
{\bf The stability condition.} We must show that the tropicalization of a family of logarithmic stable maps is stable in the combinatorial sense. It suffices to analyze the $2$-valent vertices of $f_{\trop}(G)$. Let $v$ be such a vertex. By the balancing condition, the edges emanating from $f(v)$ are contained in an affine line in $N_\RR$. Suppose that in a neighborhood of $f(v)$ on this line, $f(v)$ is not the unique intersection of $f_{\trop}(G)$ with an extended cone of $\overline \Delta$. Then $v$ and the two edges emanating from $v$ are contained in a single extended cone $\sigma$. Thus, the component $C_v$ is mapped to the locally closed stratum $V(\sigma)$. The curve $C_v$ is projective and $V(\sigma)$ is a torus, so we conclude that $C_v$ must be contracted by $f$. Since $C_v$ is only $2$-marked, this implies the map $f$ is unstable. Thus, if $f^{\trop}$ is unstable then $f$ must be unstable from which we conclude that the tropicalization of a stable family must be stable. Finally, the balancing condition follows from the structure theorem for tropicalizations~\cite[Theorem 3.3.5]{MS14}. 
\end{proof}

\begin{remark}
On the tropical side, the stability condition formulated in this text is special to toric varieties. The locally closed strata of toric varieties are affine, so maps from complete curves to the strata must be constant. As a consequence, stability of the special fiber of a family of maps can be detected at the tropical side. 
\end{remark}

\subsection{The minimal base structure} In the next section we introduce the notion of an unsaturated stable map associated to a minimal logarithmic stable map following ideas of Chen in the Deligne--Faltings I case~\cite[Section 3.7]{Che10}. For the reader's convenience, we first describe how one can understand the minimal base log structure, closely following~\cite[Construction 1.16]{GS13}.

The moduli space $\Lsm_\Gamma(X)$ carries a universal \textit{minimal logarithmic structure}, with characteristic sheaf $\mathscr Q$. Concretely, this means that given \textit{any} logarithmic stable map $\xi = [f: C\to X]$ over a logarithmic point $\spec(P\to \CC)$, there is a monoid $\mathscr Q(\xi)$, a logarithmic stable map $\xi^{\mathrm{min}}$ over $\spec(\mathscr Q(\xi)\to \CC)$, and a unique map $\mathscr Q(\xi)\to P$, such that $\xi$ is pulled back from $\xi^{\mathrm{min}}$ via $\mathscr Q(\xi)\to P$. The deformation parameters of any nodal curve $C$ give rise to a monoid $\NN^{\mathrm{nodes}}$. The monoid $\mathscr Q(\xi)$ captures the relations between these deformation parameters imposed by contact orders of the map to $X$, and the stratification on $C$ induced by this map. We now describe this more carefully.

Since we wish to describe the universal minimal base, we may and do choose a pull back along any map $P\to \NN$ of monoids, and henceforth $S$ denotes the standard logarithmic point. The first piece of data is imposed by the components of $C$ mapping to the strata of $X$. Let $\eta_i$ be a component of $C$, mapping to a stratum $V(\sigma_i)$. Assume that the integral monoid of the orbit $\sigma_i$ is $N_i$ and let $M_i = \Hom(N_i,\NN)$. The logarithmic structure at the generic point $\eta_i$ is pulled back from the base $\spec(\NN\to\CC)$. By definition, we must have a map $M_i\to \NN$, i.e., an element $v_i\in N_i$. Ranging over components, each logarithmic map over $\spec(\NN\to \CC)$ associates, to each component of $C$, a lattice point in $N$.

The next piece of data comes from the nodes of $C$. Let $q$ be a node of $C$ with branches $\eta^1_q,\eta^2_q$. Assume that the character lattice of the stratum that $q$ maps into is $M_q$. By using Kato's characterization of the logarithmic structure at the node of a logarithmically smooth curve~\cite{Kato00}, the map $f$ induces a map of characteristic monoids
\[
f^\flat_q:M_q\to P\oplus^\NN \NN^2,
\]
where $P\oplus^\NN \NN^2 = P\langle \log x, \log y\rangle / (\log x+\log y = \rho_q)$, and $\rho_q$ is the image in the characteristic of the deformation parameter of the node $C$. By~\cite[Remark 1.2]{GS13}, the monoid pushout $P\oplus^\NN \NN^2$ can be described as the set of pairs
\[
\{(p_1,p_2)\in P\times P:p_2-p_1\in \ZZ\rho_q\}.
\]
The morphism $f_q^\flat$ is equivalent to a map
\[
\varphi_q: M_q\to P\oplus^\NN \NN^2\hookrightarrow P\times P.
\]
Thus, we obtain a homomorphism
\[
u_q: M_q\to \ZZ,
\]
satisfying the relation 
\begin{equation}{\label{minimal-relation}}
(p_2-p_1)\circ \varphi_q(m) = u_q(m)\rho_q.
\end{equation}
These are the ``minimal'' requirements that are imposed by a logarithmic map, on the logarithmic structure of the base. Keeping the notation above, consider the monoid
\[
Q (\xi) = \left(\prod_{\eta_i} M_{\sigma_i}\times \prod_{q: \mathrm{node}} \NN\right)\big/R,
\]
where $R$ is the subgroup generated by the relation~(\ref{minimal-relation}). Let $\mathscr Q^{\mathrm{us}}(\xi)$ be the torsion free part of $Q(\xi)$. By the discussion above, for every logarithmic map over $\spec(P\to \CC)$, any pull back to $\spec(\NN\to \CC)$ must satisfy the above relations. By the mapping property of minimality described above, the characteristic monoids of the universal minimal log structure are the saturations $\mathscr Q(\xi)$ of $\mathscr Q^{\mathrm{us}}(\xi)$. The saturation is defined in~\cite[Chapter I, 1.2.3]{Ogu06}. The superscript $^{\mathrm{us}}$ stands for \textit{unsaturated}. 

The tropically inclined reader may benefit from dualizing this description. The equivalent information of $\Hom(Q(\xi),\RR_{\geq 0})$ is a cone of tropical curves with data determined by $\xi$. However, the second dualization required to recover $Q(\xi)$ automatically saturates it. We will need to work with the monoid before saturation in the next section.

\subsection{Unsaturated logarithmic maps}~\label{sec: coarsemaps} Let $\xi$ be a minimal logarithmic stable map over a geometric point $\underline S$
\[
\xi \ \ \ = 
\begin{tikzcd}
(C,\mathpzc M_C) \arrow{d} \arrow{r}{f} & (X,\mathpzc M_X) \\
(S,\mathpzc M_S). & 
\end{tikzcd}
\]
Since $\xi$ is minimal, the characteristic of ${\mathpzc M}_S$ is naturally isomorphic to $\mathscr Q(\xi)$. 

Let $(S,{\mathpzc M_S^{\mathrm{us}}})$ be the scheme $S$ with the logarithmic substructure generated by $\mathscr Q^{\mathrm{us}}$ in $\mathpzc M_S$. There is a moduli map $\varphi: \underline S\to \fM_{0,n+m}$ to the log smooth Artin stack of marked pre-stable curves. Let $\mathpzc M_S^{\varphi}$ be the pullback of this logarithmic structure on $\fM_{0,n+m}$ to $\underline S$. The characteristic of $\mathpzc M_S^{\varphi}$ is $\NN^{\# \ \mathrm{nodes}}$. There is natural map of characteristic monoids
\[
\overline{\mathpzc M_S^\varphi}\to \mathscr Q(\xi),
\]
factoring through $\mathscr Q^{\mathrm{us}}(\xi)$, so the map $\mathpzc{M}_S^\varphi\to \mathpzc M_S$ factors through $\mathpzc M_S^{\mathrm{us}}$. This induces a logarithmic curve $(C,\mathpzc M_C^{\mathrm{us}})\to (S,{\mathpzc M_S^{\mathrm{us}}})$. Similarly, by inspecting the characteristic, we see that the map $f^\star\mathpzc M_X\to \mathpzc M_C$ factors through $\mathpzc M_C^{\mathrm{us}}$, so we obtain an induced logarithmic map
\[
\xi^{\mathrm{us}} \ \ \ = 
\begin{tikzcd}
(C,\mathpzc M_C^{\mathrm{us}}) \arrow{d} \arrow{r}{f^{\mathrm{us}}} & (X,\mathpzc M_X) \\
(S,\mathpzc M_S^{\mathrm{us}}). & 
\end{tikzcd}
\]

\noindent

\begin{definition}
Let $\xi$ be a minimal logarithmic stable map. The logarithmic map $\xi^{\mathrm{us}}$ is called the \textbf{unsaturated logarithmic map associated to $\xi$}.
\end{definition}

\begin{remark}
In~\cite{Che10}, the notion of an unsaturated logarithmic map is referred to as a \textit{coarse map}. We prefer the term unsaturated as it avoids conflict with other notions of ``coarse'' in the subject.
\end{remark}

Given a minimal logarithmic stable map $\xi$ with discrete data $\Gamma$, it follows from~\cite[Remark 1.2]{GS13} that $\Hom(\mathscr Q(\xi),\RR_{\geq 0})$ is a cone of tropical stable maps with discrete data $\Gamma$. We refer to the combinatorial type $\Theta$ of the tropical maps parametrized by this cone as the \textit{combinatorial type} of $\xi$.

The following proposition is analogous to~\cite[Lemma 3.7.4]{Che10}. In fact, the result holds for logarithmic stable maps of any genus, with target a logarithmically smooth scheme $X$ such that the sheaf of groups associated to the characteristic $\overline{\mathpzc{M}}_X$ is globally generated, as in~\cite{GS13}.

\begin{proposition}\label{prop: coarse}
Let $\xi_1 = (C\to S,\mathpzc{M}_{S,1},f_1)$ and $\xi_2 = (C\to S,\mathpzc{M}_{S,2},f_2)$ be minimal log stable maps such that the combinatorial types of $\xi_1$ and $\xi_2$ coincide, and the underlying stable maps $\underline \xi_1$ and $\underline \xi_2$ coincide. Then, there is a canonical isomorphism of unsaturated log maps $\xi_1^{\mathrm{us}}\cong \xi_2^{\mathrm{us}}$.
\end{proposition}

\begin{proof}
Let $\mathpzc{M}_{C,1}^{\mathrm{us}}$ and $\mathpzc{M}_{C,2}^{\mathrm{us}}$ denote the unsaturated logarithmic structures on the source curves of $\xi_1$ and $\xi_2$ respectively. Since the underlying maps $\underline\xi_1$ and $\underline\xi_2$ coincide, we may pull back the logarithmic structure on $X$ via the underlying map of $\xi_i$. The curve $\underline C\to \underline S$ gives a moduli map $\varphi: \underline S\to \fM_{0,n+m}$, so we may pull back the logarithmic structure on the stack $\fM_{0,n+m}$ to $C$. Putting these maps together, we have a diagram of solid arrows:
\[
\begin{tikzcd}
& \mathpzc{M}_{C,1}^{\mathrm{us}} \arrow[dotted]{dd}{\Psi}& \\
(\underline f^{\mathrm{us}})^\star \mathpzc{M}_X\arrow[swap]{dr}{(f_2^{\mathrm{us}})^\flat}\arrow{ur}{(f_1^{\mathrm{us}})^\flat}& & \mathpzc{M}_S^\varphi \arrow{dl}{\varphi_2} \arrow[swap]{ul}{\varphi_1}\\
& \mathpzc{M}_{C,2}^{\mathrm{us}} &
\end{tikzcd}
\]
Denote by $\pi: \underline C\to \underline S$ the underlying family of curves. The underlying structures of $\xi_1$ and $\xi_2$ coincide, so to construct $\Psi$ it suffices to construct a map $\psi: \pi^\star\mathpzc{M}_{S,1}^{\mathrm{us}}\to \pi^\star\mathpzc{M}_{S,2}^{\mathrm{us}}$. Choose a chart $\mathscr Q^{\mathrm{us}}(\xi_1)\to \mathpzc{M}_{C,1}^{\mathrm{us}}$. Recall that we have a presentation of the unsaturated characteristic
\[
\left(\prod_{\eta_i} M_{\sigma_i}\times \prod_{q: \mathrm{node}} \NN\right)\to \mathscr Q^{\mathrm{us}}(\xi_j).
\]
Let $\overline e$ be an element of $\mathpzc{M}^{\mathrm{us}}_{C,1}$ corresponding to a node. There is an element $e\in \mathpzc{M}^\varphi_S$ lifting $\overline e$ via $\varphi_1$. Define $\psi(\overline e) = \varphi_2(e)$. Now consider $\overline v\in \mathpzc{M}^{\mathrm{us}}_{C,1}$ corresponding to the generic point of a component of $C$. By construction we have $v\in M_{\sigma_i}$ mapping to $\overline v$. Define $\psi(\overline v) = (f_2^{\mathrm{us}})^\flat(v)$. The combinatorial type of $\xi_i$ is $\Hom(\mathscr Q^{\mathrm{us}}(\xi_j),\RR_{\geq 0})$ and since the combinatorial types of $\xi_1$ and $\xi_2$ coincide, this map descends via $\left(\prod_{\eta_i} M_{\sigma_i}\times \prod_{q: \mathrm{node}} \NN\right)\to \mathscr Q^{\mathrm{us}}(\xi)$, so $\Psi$ is well defined. Interchanging the roles of the log structures $\mathpzc{M}_{C,1}$ and $\mathpzc{M}_{C,2}$, we obtain an inverse map to $\Psi$ defining the desired isomorphism. The result follows.
\end{proof}

%

\subsection{Proof of Theorem~\ref{thm: skeleton}}

We divide the proof into three steps. First we show that the set theoretic map $\trop$ factors as 
\[
\begin{tikzcd}
\Lsm^{\an}_\Gamma(X) \arrow[swap]{dr}{\bm p} \arrow{rr}{\trop} & & \Tsm_\Gamma(\Delta) \\
& \overline \fS \arrow[swap]{ur}{\trop_\fS}. & \\
\end{tikzcd}
\]
We then show that $\trop_\fS$ is surjective, and then injective. Since both $\overline \fS$ and $\Tsm_\Gamma(\Delta)$ are canonical compactifications associated to cone complexes, it suffices to prove the commutativity of the diagram above replacing $\Lsm_\Gamma^{\an}(X)$ with $\Lsm_\Gamma^{\an,\circ}(X)$, and $\Tsm_\Gamma(\Delta)$ and $\overline \fS$ with their interior cone complexes $\Tsm^\circ_\Gamma(\Delta)$ and $\fS$ respectively. 

\noindent
{\bf Step I. The tropicalization map factors.} A point $x\in \Lsm^{\circ,\an}_\Gamma(X)$ can be represented by a map $\spec(K)\to \Lsm^\circ_\Gamma(X)$, where $K$ is a valued field extending $\CC$. Since $\Lsm_\Gamma(X)$ is proper as a stack~\cite[Theorem 0.2]{GS13}, after replacing $K$ with a ramified extension with valuation ring $R$, we obtain a map $\spec(R)\to \Lsm_\Gamma(X)$. This yields a family of logarithmic stable maps
\[
\begin{tikzcd}
C\arrow{d}\arrow{r} & \mathscr C \arrow{d}\arrow{r}{f} & X \\
\spec(K)\arrow{r} & \spec(R) & 
\end{tikzcd}
\]
There is an associated tropical map $f_{\trop}$ as constructed in Section~\ref{sec: pointwise-trop}, and we interpret $[f_{\trop}]$ as the point $\trop(x)\in \Tsm^\circ_\Gamma(\Delta)$. 

We now compute $\bm p(x)$. Let $\overline x$ be the image of the closed point of the moduli map $\spec(R)\to \Lsm_\Gamma(X)$ and let $\Theta$ be the combinatorial type of the logarithmic map associated to $\overline x$. By logarithmic smoothness, proved in Proposition~\ref{prop: logsmooth}, there is an \'etale local neighborhood $U$ of $\overline x$ isomorphic to  $\spec(\CC\llbracket Q\rrbracket)\times \mathbb G_m^r$, where $Q$ is the minimal monoid at $\overline x$. By~\cite[Proposition 6.2]{U13}, the skeleton of $(U\cap \Lsm^\circ_\Gamma(X))^\beth$ is $\Hom(Q,\RR_{\geq 0})$. From~\cite[Remark 1.21]{GS13} this cone is canonically identified with $\sigma_\Theta$, the cone of tropical stable maps with combinatorial type $\Theta$. 

Now, suppose $x\in \fS\subset \Lsm^{\an,\circ}_\Gamma(X)$ is a point of the skeleton. Any point of the analytic space gives rise to a valuation. In this case, $x$ furnishes a valuation for each monomial function in $Q$, corresponding to products of deformation parameters of nodes. This produces a monoid map $\varphi: Q\to \RR_{\geq 0}$. The point $x$ is then represented by the composite
\[
\CC[U]\xrightarrow{\val_\varphi} \CC\llbracket \RR_{\geq 0}\rrbracket \xrightarrow{\mathrm{ord}} \RR_{\geq 0},
\]
where $\val_\varphi$ is the valuation induced by the map $\varphi$ above and $\mathrm{ord}$ is the standard valuation on $\CC\llbracket \RR_{\geq 0}\rrbracket$. It follows that $\trop_\fS$ maps $\fS(U^\beth)$ isomorphically onto the cone $\sigma_\Theta$, and we conclude that
\[
\trop(x) = \trop_\fS\circ \bm p(x),
\]
so the map factors as claimed. Upon restriction to any face of $\fS$ the same calculation shows that $\trop_\fS$ is an isomorphism onto its image, i.e. $\trop_\fS$ is a face morphism. 

\noindent
{\bf Step II. The map $\trop_\fS$ is surjective} Let $[f_\trop:G \to \overline\Delta]$ be a tropical stable map from a smooth tropical curve of combinatorial type $\Theta$. Let $\underline{\mathscr C}_0$ be a marked rational nodal curve with dual graph $\underline G$. We wish to build a logarithmic map $\underline f_0: \underline{\mathscr C}_0\to X$ with discrete data given by $f_\trop$. Suppose $v$ is a vertex of $G$ mapping to the relative interior of $\sigma_v$. There is a natural map induced by $f_\trop$,
\[
f^v_\trop: G_v\to N(\sigma_v) = N/\sigma_v,
\]
where $G_v$ is the star of $G$ around $v$. The graph $G_v$ has a single vertex $v$ and $r$ infinite outgoing edges in bijection with the outgoing edges of $G$ at $v$. The map $f_\trop^v$ sends the vertex $v$ to $0\in N(\sigma_v)$, and maps each infinite edge $e_i$ to $N(\sigma_v)$ with expansion factor and edge direction given by the contact order $c_{e_i}\in N/\sigma_v^{\mathrm{gp}}$. We now construct a logarithmic map to $\overline V(\sigma_v)$ that is dual to $f^v_\trop$. After a toric modification $\widetilde V\to \overline V(\sigma_v)$, we may assume that the contact orders of $f^v_\trop$ are torically transverse. $f^v_\trop$ is determines the discrete data of a map to $\widetilde V$. 

We first construct a marked rational curve dual to $C_v$ with a logarithmic to $\widetilde V$ with contact orders prescribed by $f^v_\trop$. Since the contact order is torically transverse, the contact order is determined by the order of tangency with the boundary of $\widetilde V$, so such a map always exists by Proposition~\ref{prop: lsm-m0n}. We obtain a logarithmic map to the original stratum $V(\sigma_v)$ by composing the underlying map with $\widetilde V\to V(\sigma_v)$, and pushing forward the logarithmic structure using the results of~\cite[Appendix B]{AMW12}. We have constructed the desired map
\[
f^v_0:C_v\to \overline V(\sigma_v)
\]
Ranging over all vertices $v$, it is now straightforward to check that these maps $f^v_0$ glue to form a logarithmic map
\[
f_0:{\mathscr C_0}\to X.
\]
As argued in~\cite[Section 3.3]{R15a}, by the existence and logarithmic smoothness of $\Lsm_\Gamma(X)$, we obtain a family
\begin{equation}\label{curve-over-C[Q]}
\begin{tikzcd}
\mathscr C' \arrow{r} \arrow{d} & X \\
\spec(Q\to \CC\llbracket Q\rrbracket). & \\
\end{tikzcd}
\end{equation}
The point $[f_\trop]$ is equivalent to a monoid homomorphism $Q\to \RR_{\geq 0}$ and thus a homomorphism $\varphi: \CC\llbracket Q\rrbracket \to \CC\llbracket \RR_{\geq 0}\rrbracket$. Pulling back along $\varphi$, we obtain a family
\[
\begin{tikzcd}
\mathscr C \arrow{r}{f} \arrow{d} & X \\
\spec(\CC\llbracket \RR_{\geq 0}\rrbracket). & \\
\end{tikzcd}
\]
One can check as in~\cite[Section 3.3]{R15a} that the tropicalization of the family of maps $f$, in the sense of Section~\ref{sec: pointwise-trop}, is precisely $f_\trop$. By construction, $[f]$ is a valuation in the monomial coordinates induced by the monoid $Q$, and thus is a point of the space $ \fS\subset\Lsm^{\an,\circ}_\Gamma(X)$. This proves that $\trop_\fS$ is surjective.

\noindent
{\bf Step III. The map $\trop_\fS$ is injective.} Let $W$ be the locus in $\Lsm_\Gamma(X)$ parametrizing maps of a fixed combinatorial type $\Theta$. Let $\underline G$ be the dual graph of the source curve of a generic map in $W$. As in Step II, the underlying maps parametrized by $W$ can be formed by gluing logarithmic stable maps parametrized by the spaces $\Lsm^\circ_{\Gamma(v)}(\overline V(\sigma_v))$, such that the two attaching points corresponding to a node of $\mathscr C_0$ map to the same point in $X$. It follows from the explicit description of $\Lsm^\circ_{\Gamma(v)}(\overline V(\sigma_v))$ in Lemma~\ref{lem: fixed-boundary} and Proposition~\ref{prop: irred} that this space of underlying maps associated to logarithmic maps parametrized by $W$ is irreducible.

By the irreducibility of the space of underlying maps of $W$, to prove that $\trop_\fS$ is injective, it suffices to show that given two logarithmic stable maps $\xi_1 = [f_1:C_1\to X]$ and $\xi_2 = [f_2:C_2\to X]$ with the same underlying structure and combinatorial type $\Theta$, $\xi_1 = \xi_2$. In turn, by applying Proposition~\ref{prop: coarse} it suffices to show that the characteristic monoid $\mathscr Q^{\mathrm{us}}(\xi_i)$ is already saturated. Recall that this unsaturated characteristic was defined as the torsion free part of the quotient
\[
Q (\xi) = \left(\prod_{\eta_i} M_{\sigma_i}\times \prod_{q: \mathrm{node}} \NN\right){\Huge/}R.
\]
There is a relation imposed by each node of the source curve $C$ of $\xi$, i.e. by each edge in the dual graph $C$. For each edge $e$ incident to vertices $v$ and $u$, this relation is of the form
\begin{equation}\label{relation}
v = u+w_e\cdot e,
\end{equation}
where $w_e$ is the expansion factor along $e$. 

Let $v_0$ be any non-leaf vertex of the dual graph of $C$. Since the dual graph of $C$ is a tree, given any other vertex $v_i$, there exists a unique path between $v_0$ and $v_i$. By iteratively applying the relation (\ref{relation}) we may uniquely write
\[
v_i = v_0+\sum \pm w_{jk} e_{jk}.
\]
Let $\sigma$ be the cone to which the root vertex of the dual graph of $C$ maps to (i.e. the vertex supporting the first marked point of trivial contact order). It follows that $\mathscr Q^{\mathrm{us}}(\xi)$ is already saturated and we conclude that $\xi_1\cong \xi_2$.
\qed

\begin{remark}\label{rem: non-normality}
The fact that $\mathscr Q^{\mathrm{us}}(\xi)$ is already saturated relies crucially on the fact that we work with genus $0$ stable maps, as the following example illustrates. Consider two rational curves $C_1$ and $C_2$, each mapping to $\PP^1$ by a degree $5$ map, totally ramified over $0$ and with ramification $(2,3)$ over $\infty$. Glue the two maps $C_i\to \PP^1$ by gluing the source curves along the ramified preimages of $\infty$, and gluing the target curves along $\infty$. The resulting glued map $C\to \PP^1\sqcup_\infty \PP^1$ is logarithmic stable map, see Figure~\ref{fig: deg-5-cover}. Consider the map $C\to \PP^1$ obtained by composing the glued map $C\to \PP^1\sqcup_\infty \PP^1$ with the map contracting the second component of the target. The unsaturated minimal characteristic of this map is the quotient of 
\[
\langle v_1,e_1,e_2\rangle
\]
by $v_1 = 2e_1$ and $v_1 = 3e_2$. The unsaturated characteristic is isomorphic to the submonoid of $\NN$ generated by $2$ and $3$ which is not saturated. We refer to~\cite[Section 4.2]{CMR14a} for details.\newpage

\begin{figure}[h!]
\includegraphics{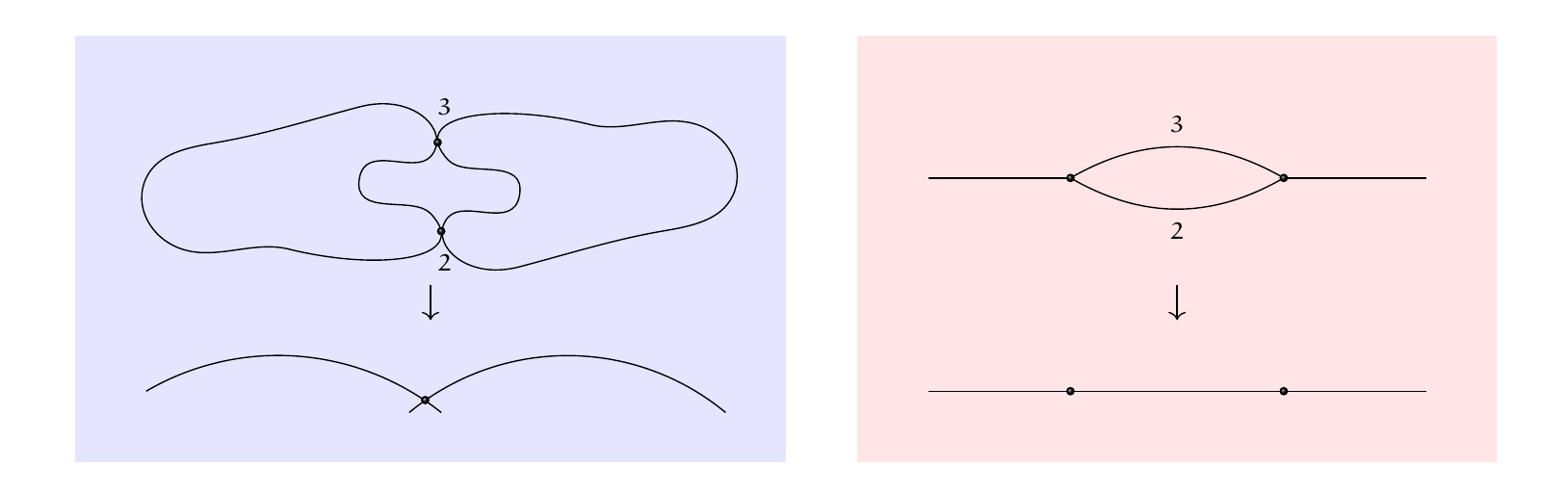}
\caption{The figure on the left depicts a nodal curve of genus $1$ covering a nodal rational curve with degree $5$. On the right we depict the tropical dual graphs.}
\label{fig: deg-5-cover}
\end{figure}
\end{remark}

%


\section{Compactification in a toric variety}

\subsection{Extended example: The logarithmic dual projective plane}\label{sec: extended-example} We illustrate the tropical and algebraic spaces of logarithmic stable maps in an elementary example that illustrates the approach to Theorem~\ref{thm: trop-comp}. 

Fix $X = \PP^2$ with its toric structure. Consider maps from rational curves of degree $1$, with contact order $1$ with each toric boundary. The moduli space $\Lsm_\Gamma(\PP^2)$ is $2$-dimensional. Indeed it is birational to the dual projective plane $\check{\PP}^2$. These two spaces are not isomorphic, since the coordinate lines $V(X_i)$ in $\PP^2$ does not meet itself transversely. The main takeaway of this section is that $\Lsm_\Gamma(\PP^2)$ admits a map to $\check{\PP}^2$ that is a toroidal modification. This toroidal modification is induced by the map on tropical fans $\Tsm^\circ_\Gamma(\Delta_{\PP^2})\to\Delta_{\check{\PP}^2}$. 

\noindent
{\bf (Algebraic side)} The moduli space $\Lsm_\Gamma(\PP^2)$ is logarithmically smooth and stratified into locally closed strata, which we now describe. Let $\mathpzc \overline {\mathpzc M}_L$ denote the logarithmic structure sheaf of $\Lsm_\Gamma(\PP^2)$. There is a dense open set $\Lsm_\Gamma^\circ(\PP^2)$ where the characteristic $\overline{\mathpzc M}_L = 0$. Let $S$ be a geometric point in this locus. Pulling back the universal family, we obtain a map $[f:C\to \PP^2]$. The curve $C$ is smooth, with marked points $p_1,p_2,p_3$. The image of  $C\setminus \{p_i\}$ lies in the dense torus $T$, and $p_i$ is mapped to the $i^{\mathrm{th}}$ boundary divisor $D_i$, with intersection multiplicity $1$.

There is a locally closed stratum in $\Lsm_\Gamma(\PP^2)$ where $\overline{\mathpzc M}_L$ is isomorphic to $\NN$,  described as follows. The marked curve $(C,p_1,p_2,p_3)$ maps isomorphically onto a boundary divisor $D_1$ of the target $\PP^2$. The stalk of the characteristic monoid at the generic point of $C$ is $\NN$. The points $p_2$ and $p_3$ are mapped to $D_1\cap D_2$ and $D_1\cap D_3$ respectively. The point $p_1$ can map to any point in the dense torus of the toric subvariety $D_1$, so this locus is $1$-dimensional. Permuting the roles of $p_1,p_2,p_3$ we obtain $3$ such $1$-dimensional strata. There is another type of locally closed stratum where $\overline{\mathpzc M}_L$ is isomorphic to $\NN$. This consists of maps from a nodal curve $C = C_1\cup C_2$, where $C_1$ carries the marking $p_1$ and $C_2$ carries the markings $p_2$ and $p_3$. The class $f_\star[C]$ is the class of a line, so $f$ must contract a component. By stability it must contract $C_2$. Since $C_2$ carries the markings $p_2$ and $p_3$, in order to satisfy the contact order, $C_2$ must contract to the torus fixed point $D_2\cap D_3$. Thus, the curve $C_1$ passes through the point $D_2\cap D_3$. Any two distinct lines meet in $\PP^2$, so the image of $C_1$ can be any line through $D_2\cap D_3$, and will automatically satisfy the contact order at $D_1$. The monoid at the generic point of $C_2$ is $\NN^2$, and the monoid is trivial along $C_1$. Maps parametrized by this locus are in natural bijection with lines in $\PP^2$ passing through a given point $D_2\cap D_3$, so we obtain $3$ more $1$-dimensional strata.

There are six points at which the sheaf $\overline{\mathpzc M}_L$ is $\NN^2$, and these lie in the closures of the strata described above. When $C = C_1\cup C_2$ and the map $f$ contracts $C_2$ as above, the line $C_1$ could map isomorphically onto either $D_2$ or $D_3$. It is not hard to enumerate the $5$ remaining $0$-dimensional strata, and we leave this to an interested reader. 

Finally, observe that we may send each map $[f:C\to \PP^2]$ to the image curve $f(C)$, which is a line in $\PP^2$, and thus a point of $\check{\PP}^2$. This gives us a logarithmic proper birational map
\[
\gamma: \Lsm_\Gamma(\PP^2)\to \check{\PP}^2.
\]
From the description above, $\Lsm_\Gamma^\circ(\PP^2)$ maps isomorphically onto $\check{T}$, the dense torus of $\check{\PP}^2$. The space $\check{\PP}^2$ is a logarithmically smooth moduli space of lines in $\PP^2$, stratified by the intersection type of each line with the toric boundary of $\PP^2$. We see from the description of $\Lsm_\Gamma(\PP^2)$ that each torus fixed point of $\check{\PP}^2$ is blown up by $\gamma$. In summary we have proved the following ``toy version'' of Theorem~\ref{thm: trop-comp}.

\begin{theorem}\label{toytheorem}
The logarithmic dual projective plane $\Lsm_\Gamma(\PP^2)$ is isomorphic to $\mathrm{Bl}_{3 \ \! \mathrm{pts}} \check{\PP}^2$. 
\end{theorem}

\noindent
{\bf (Tropical side)} We fix a presentation of $\Delta_{\PP^2}$. Its three one-dimensional cones are spanned by primitives $u_1,u_2,u_3\in \ZZ^2$ such that $\sum u_i = 0$. Two dimensional cones are spanned by distinct pairs $\langle u_j,u_k\rangle$.

Let $[f:G\to \overline \Delta_{\PP^2}]$ be a map parametrized by a point of $\Tsm^\circ_\Gamma(\Delta_{\PP^2})$. The underlying set of the image of $f$ is a trivalent graph, with a single trivalent ``root'' vertex, and $3$ infinite edges $e_1,e_2,e_3$. Thus, the set underlying $\Tsm_\Gamma^\circ(\Delta_{\PP^2})$ is in bijection with $|\Delta|\cong \RR^2$, where the bijection takes a map to the image of the root vertex in $\Delta$. It remains to determine the fan structure on $\Tsm_\Gamma(\Delta_{\PP^2})$. 

We have thus far ignored possible divalent vertices on $G$. Recall that we require that a tropical stable map sends each edge of $G$ to a single cone of $\Delta_{\PP^2}$. Given a position in $|\Delta_{\PP^2}|$ for the root vertex $v$ of $G$, there is a unique (possibly trivial) subdivision of $G$ that makes this subdivision hold. We describe the cones resulting from this subdivision.

There is a $0$-dimensional cone of  $\Tsm^\circ_\Gamma(\Delta_{\PP^2})$ corresponding to the trivalent graph $G$, mapping to $0\in \Delta_{\PP^2}$, and the $3$ infinite edges of $G$ mapping along the one-skeleton of $\Delta_{\PP^2}$. 

There is a $1$-dimensional cone of $\Tsm^\circ_\Gamma(\Delta_{\PP^2})$ parametrizing maps where $v$ is mapped to a $1$-dimensional ray in the fan $\Delta_{\PP^2}$. We observe by inspection that each edge of $G$ maps to a single cone of $\Delta_{\PP^2}$ and the stability condition clearly holds. By permuting the roles of the rays of the one skeleton, we obtain $3$ such $1$-dimensional cones in $\Tsm^\circ_\Gamma(\Delta_{\PP^2})$. There is another type of $1$-dimensional cone, consisting of maps where $v$ is sent to the top dimensional cone $\langle u_2,u_3\rangle$ of $\Delta_{\PP^2}$, and the ray $e_1$ passes through the $0$-cone of $\Delta_{\PP^2}$. In this case, we are forced to subdivide $G$ and its image along the vertex. This yields a bounded edge $e_b$, and is a distinct combinatorial type, whose associated cone is $1$-dimensional. Points of this cone are obtained by varying the length of $e_b$. Permuting the roles of $e_i$ we obtain three more $1$-dimensional cones in $\Tsm^\circ_\Gamma(\Delta_{\PP^2})$.

Finally, we describe the six $2$-dimensional cones in $\Tsm^\circ_\Gamma(\Delta_{\PP^2})$. Assume that $v\in G$ maps to the interior of the cone $\langle u_1,u_2\rangle$ in $\Delta_{\PP^2}$. The edges $e_1$ and $e_2$ are parallel to the rays $\langle u_1\rangle$ and $\langle u_2\rangle$ respectively. The final edge $e_3$ must intersect one of the cones $\langle u_i\rangle$ for $i = 1,2$. Assume that $e_3$ intersects $\langle u_1\rangle$. In order to have a map of polyhedral complexes, we must subdivide $G$ and its image at this intersection point. This subdivision changes the combinatorial type of this tropical stable map, and creates a new edge $e_b$ on $G$. There is a $2$-dimensional cone of maps of this combinatorial type, with coordinates given by the length of the bounded edge $e_b$ and the intersection point of $f(G)$ with $\langle u_1\rangle$. Similarly, if $e_3$ intersects the ray $\langle u_2\rangle$, we obtain a two dimensional cone in $\Tsm^\circ_\Gamma(\Delta_{\PP^2})$. These cones glue along the ray where $e_3$ passes through $0$.

We conclude by observing that the tropical and algebraic pictures are compatible. See Figure~\ref{eg: tropical-moduli}.

\begin{theorem}
The space of tropical stable maps of degree $1$ to $\Delta_{\PP^2}$ with contact order $1$ for each marked infinite ray is the compactified fan of the toric variety $\mathrm{Bl}_{3 \ \! \mathrm{pts}} \check{\PP}^2$.
\end{theorem}

\begin{figure}[h!]
\begin{tikzpicture}

\fill[white!70!blue, path fading=north] (0,0)--(0,2)--(45:2.2) -- cycle;
\fill[white!70!blue, path fading=east] (0,0)--(45:2.2)--(2,0) -- cycle;
\fill[white!70!blue, path fading=south] (0,0)--(2,0)--(0,-2) -- cycle;
\fill[white!70!blue, path fading=south] (0,0)--(0,-2)--(45:-2.2) -- cycle;
\fill[white!70!blue, path fading=west] (0,0)--(45:-2.2)--(-2,0) -- cycle;
\fill[white!70!blue, path fading=north] (0,0)--(-2,0)--(0,2) -- cycle;

\begin{scope}[shift = {(5,0)}]
\fill[pink, path fading=north] (0,0)--(0,2)--(2,2)--(2,0) -- cycle;
\fill[pink, path fading=south] (0,0)--(2,0)--(2,-1.414)--(-1.414,-1.414) -- cycle;
\fill[pink, path fading=west] (0,0)--(0,2)--(-1.414,2)--(-1.414,-1.414) -- cycle;
\end{scope}

\draw [ball color = black!40!blue] (0,0) circle (0.5mm);

\draw[->,black!40!blue] (0,0)--(0,2);
\draw[->,black!40!blue] (0,0)--(2,0);
\draw[->,black!40!blue] (0,0)--(0,-2);
\draw[->,black!40!blue] (0,0)--(-2,0);
\draw[->,black!40!blue] (0,0)--(45:2.2);
\draw[->,black!40!blue] (0,0)--(45:-2.2);

\draw[->,red] (5,0)--(5,2);
\draw[->,red] (5,0)--(7,0);
\draw[->,red] (5,0)--(3.686,-1.414);
\draw [ball color=black] (1,1.5) circle (0.5mm);

\node at (1,1.75) {\tiny $P$};

\draw [->,
line join=round,
decorate, decoration={
    zigzag,
    segment length=4,
    amplitude=.8,post=lineto,
    post length=2pt
}]  (1.2,1.45) -- (3.5,1);

\draw (7.5,1.5)--(6,1.5)--(6,3);
\draw (6,1.5)--(4.5,0);

\draw [ball color=black] (6,1.5) circle (0.5mm);
\draw [ball color=black] (5,0.5) circle (0.5mm);
\draw [ball color=red] (5,0) circle (0.5mm);
\end{tikzpicture}
\caption{On the left is depicted the open moduli space $\Tsm_\Gamma^\circ(\Delta_{\PP^2})$ of tropical stable maps of degree $1$ to $\Delta_{\PP^2}$. On the right is depicted the map parametrized by the point in  $P\in\Tsm_\Gamma^\circ(\Delta_{\PP^2})$. The coordinates of $P$ determine the coordinates of the root vertex on the right.}
\label{eg: tropical-moduli}
\end{figure}
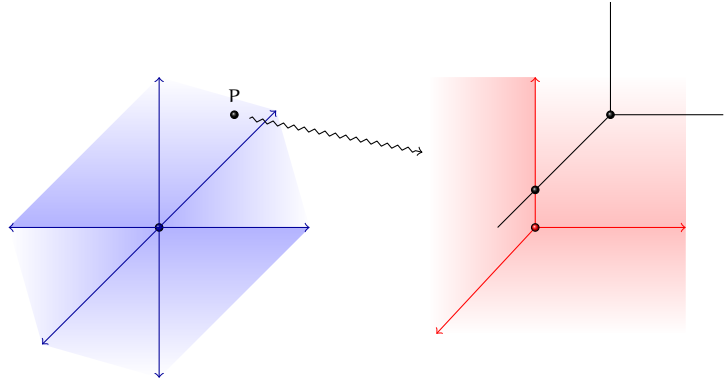

\subsection{A fan structure on the space of tropical maps} 
Throughout this section we assume that our logarithmic stable maps have at least one marked point with trivial contact order, i.e. $m\geq 1$. We distinguish such a marking $p_1$ and let $ev_1$ be the associated evaluation morphism. Fix the dimension of $X$ to be $r$. The tropical moduli space $M_{0,n}^\trop$ has the natural structure of a balanced fan in a vector space, by expressing it as a quotient of the Speyer--Sturmfels tropical Grassmannian of planes~\cite{SS04a}. In~\cite{GKM07}, Gathmann, Kerber, and Markwig use this to describe fan structures on spaces $\mathscr M^{\mathrm{lab}}_{0,n,\trop}(|\Delta|,\Gamma)$ of parametrized rational tropical curves in $|\Delta|\cong \RR^r$ with prescribed unbounded edge directions given by the contact order $c$. See~\cite[Definition 4.1]{GKM07}.

\begin{proposition}\label{prop: trop-embedding}
There is an embedding of the polyhedral complex $\Tsm^\circ_\Gamma(\Delta)$
\begin{equation}\label{trop-embed-eqn}
\jmath: \Tsm^\circ_\Gamma(\Delta)\hookrightarrow \RR^k,
\end{equation}
as a fan, for $k = \dim \Delta+\binom{n+m}{2}-n-m$. 
\end{proposition}

\begin{proof}
Let $[f:G\to \overline \Delta]$ be a tropical stable map representing a point of $\Tsm^\circ_\Gamma(\Delta)$. Deleting the infinite points of $G$ that have a non-zero contact order, we obtain a map from a non-compact graph $\widetilde G\to \Delta$. After straightening all $2$-valent vertices of $G$, and forgetting the fan structure on $|\Delta|$ we a parametrized tropical curve $\widehat G\to |\Delta|$, as defined by Gathmann, Kerber, and Markwig~\cite[Definition 4.1]{GKM07}. Conversely, given a parametrized tropical curve $|\Delta|$, we may uniquely compactify $\widehat G$, and minimally subdivide to obtain a tropical stable map. This gives us a forgetful morphism
\[
\Phi:\Tsm^\circ_\Gamma(\Delta)\to \mathscr M^{\mathrm{lab}}_{0,n,\trop}(|\Delta|,\Gamma),
\]
that is a bijection on underlying sets. Moreover, $\widehat G\to |\Delta|$ and $G\to \Delta$ differ only by a subdivision, and compactification of infinite edges, so the map $\Phi$ is a refinement of fans. The proposition now follows from the corresponding result for the space $\mathscr M^{\mathrm{lab}}_{0,n,\trop}(|\Delta|,\Gamma)$, see~\cite[Proposition 4.7]{GKM07}.
\end{proof}

Let $\cT$ denote the fan structure on the resulting embedding of the cone complex $\Tsm^\circ_\Gamma(\Delta)$. The algebraic moduli space $M_{0,n+m}$ is very affine and embeds into a torus. We use Kapranov's embedding~\cite{Kap93}. Given a point of $M_{0,n+m}$ one can associate the $2$-plane in $\CC^{n+m}$ spanned by the $2\times (n+m)$ matrix describing $p$ in any coordinates. This association is well defined up to rescaling the coordinates of $\CC^{n+m}$, and gives an embedding 
\[
M_{0,n+m}\hookrightarrow G(2,n+m)\!\sslash\!T',
\] 
where $T'$ is the dense torus of $\CC^{n+m}$. By embedding $G(2,n+m)$ by the Pl\"ucker map
\[
\iota: M_{0,n+m}\hookrightarrow \PP^{\binom{n+m}{2}-1}\!\sslash\!T'
\] 
with image contained in the dense torus $T''$ of the Chow quotient $\PP^{\binom{n+m}{2}-1}\!\sslash\!T'$.\\

\noindent
Let $\mu: \Lsm^\circ_\Gamma(X)\to M_{0,n+m}$ be the morphism sending a map to its marked source curve. Recall by Proposition~\ref{prop: lsm-m0n} that $\mu\times ev_1$ is an isomorphism. 

\begin{proposition}\label{prop: a-gross-proposition}
The tropicalization of $\Lsm^\circ_\Gamma(X)$ under the embedding $\iota\times {\bm{I\!d}}$ coincides set theoretically with the image of $\Tsm^\circ_\Gamma(\Delta)$ in the embedding (\ref{trop-embed-eqn}) above.
\end{proposition}

\begin{proof}
By~\cite[Section 5]{Tev07}, the tropicalization of $M_{0,n+m}$ in this embedding is $M_{0,n}^{\trop}$, so the tropicalization of $\Lsm^\circ_\Gamma(X)$ is $M_{0,n+m}\times |\Delta|$. The result now follows from Proposition~\ref{prop: trop-embedding} above and~\cite[Proposition 4.7]{GKM07}. 
\end{proof}

The construction above embeds $M_{0,n}$ in the dense torus $T''$ of a toric variety. We deduce Theorem~\ref{thm: trop-comp} by using Tevelev's tropical compactification of $\overline M_{0,n}$ as a starting point. We consider $M_{0,n}^\trop$ as embedded in the vector space $N''_\RR$, where $N''$ is the cocharacter lattice of the torus $T''$. Let $\cF_n$ denote the resulting fan structure on $M^\trop_{0,n}$. The following result is due to Tevelev~\cite[Section 5]{Tev07}, building on earlier ideas of Kapranov~\cite{Kap93}. 

\begin{theorem}
The closure of the image of $M_{0,n}$ under the composite
\[
M_{0,n}\hookrightarrow T''\hookrightarrow X(\cF_n)
\]
is isomorphic to $\overline M_{0,n}$. 
\end{theorem}

The proof of Theorem~\ref{thm: trop-comp} requires the use of the \textit{Artin fan} associated to a logarithmic scheme $X$, which we briefly recall. See~\cite{ACMUW,ACMW, AW,U-thesis} for details. An Artin fan is a logarithmic algebraic stack that is logarithmically \'etale over a point.  By work of Olsson, a logarithmic structure on a scheme $X$ defines a tautological map $\underline X\to \mathrm{LOG}$, where $\mathrm{LOG}$ is Olsson's stack of logarithmic structures~\cite{Ols03}. Maps to $\mathrm{LOG}$ parametrize logarithmic structures on the source. In~\cite{AW}, Abramovich and Wise show that given a logarithmically smooth scheme $X$, there is an initial factorization of the tautological map $X\to \mathrm{LOG}$ as
\[
\begin{tikzcd}
X \arrow{rr} \arrow{dr} & & \mathrm{LOG} \\
& \mathpzc A_X \arrow{ur} &
\end{tikzcd}
\]
such that the map $\mathpzc{A}_X\to \mathrm{LOG}$ is \'etale, representable, and strict. Loosely speaking, $\mathpzc{A}_X$ is the image of $X$ in the stack $\mathrm{LOG}$, and thus picks out the part of the stack that is  ``relevant'' to $X$. 

\subsection{Proof of Theorem~\ref{thm: trop-comp}} Consider the forgetful morphism
\[
\mu \times ev_1: \Lsm_\Gamma(X)\to \overline M_{0,n+m}\times X,
\]
given by stabilizing the source curve in the first factor. In order to prove the result we need to show that this morphism is a toroidal modification (i.e. a logarithmically etale modification). We will do so by appealing to Abramovich and Wise's characterization of such morphisms by pullbacks of morphisms between Artin fans.

This morphism is clearly proper. It is birational by Proposition~\ref{prop: a-gross-proposition} above. By applying~\cite[Theorem 1.2.1]{ACP}, the fan of the toroidal embedding $\overline M_{0,n+m}\times X$ is given by $M^{\trop}_{0,n+m}\times \Delta$. Similarly, by Theorem~\ref{thm: skeleton}, the fan of the toroidal embedding $\Lsm_\Gamma(X)$ is given by $\Tsm^\circ_\Gamma(\Delta)$. The map of fans induced by $\mu\times ev_1$
\[
\Tsm_\Gamma^\circ(\Delta)\to M_{0,n+m}^{\trop}\times \Delta
\]
is a bijection on underlying sets. A cone $(\sigma,\delta)\in M_{0,n+m}^{\trop}\times \Delta$ parametrizes the stabilized tropical source curve and the position of the distinguished root vertex. Given a point $([G],v)\in(\sigma,\delta)$, we subdivide $G$ at its intersection points with the cones of $\Delta$, without changing the lengths of the bounded edges, and uniquely obtain a point of $\Tsm_\Gamma^\circ(\Delta)$. It follows that $\Tsm_\Gamma^\circ(\Delta)\to M_{0,n+m}^{\trop}\times \Delta$ is a refinement of fans. Thus, in local toric charts, $\mu\times ev_1$ is a toric modification, so globally it is a logarithmically \'etale modification. By~\cite[Corollary 2.5.6]{AW}, this logarithmically \'etale modification is determined by the Cartesian square
\[
\begin{tikzcd}
\Lsm_\Gamma(X) \arrow{d}\arrow{r} & \overline M_{0,n+m}\times X \arrow{d}\\
\mathpzc{A}_L \arrow{r} & \mathpzc{A}_M,
\end{tikzcd}
\]
where $\mathpzc{A}_L$ and $\mathpzc{A}_M$ are Artin fans. It follows from Theorem~\ref{thm: skeleton} and~\cite[Theorem V.1.1]{U-thesis}, there is a canonical isomorphism
\[
\mathpzc A_{L}^\beth\xrightarrow{\sim} \Tsm_\Gamma(\Delta),
\]
so $\mathpzc A_L$ is the Artin fan of $\Lsm_\Gamma(X)$. All logarithmic structures are defined on the Zariski site, so the Artin fan construction is functorial~\cite[Section 10.3]{Kato94}. By its construction, the quotient of the toric variety $X(\cT)$ (resp. $X(\cF_{n+m}\times \Delta)$) by its dense torus is canonically identified with $\mathpzc A_L$ (resp. $\mathpzc A_M$). The map of fans 
\[
\cT \to \cF_{n+m} \times \Delta
\] 
gives rise to a toric modification $X(\cT)$ of the toric variety $X(\cF_{n+m})\times X(\Delta)$. Let $Y$ be the proper transform of $\overline M_{0,n+m}\times X(\Delta)$ via this toric modification. Since $\overline M_{0,n+m}$ is transverse to all toric strata of $X(\cF_{n+m})$, the inclusion $Y\hookrightarrow X(\cT)$ is a strict morphism, so there is a natural map $Y\to \mathpzc A_L$.
By universal property of the pullback, we have a map $\varphi: \Lsm_\Gamma(X)\to Y$ that is proper and birational. Again, since $Y$ is transverse to all toric strata $\varphi$ is strict. Both $\Lsm_\Gamma(X)$ and $Y$ are logarithmically \'etale over $\mathpzc{A}_M$, so the map $\varphi$ is \'etale, and hence an isomorphism by Zariski's main theorem. The fan $\cT$ furnishes the fan structure $\Sigma_\Gamma$ on $\trop(\Lsm^\circ_\Gamma(X))$ as claimed. 
\qed

\begin{remark}\label{rem: assumption-on-m}
The assumption of $m\geq 1$ is required to realize the space $\Lsm^\circ_\Gamma(X)$ as the product $M_{0,n+m}\times T$. When $m = 0$, one does not have an evaluation morphism to $T$. Nonetheless the fibers of the morphism
\[
\Lsm_\Gamma^\circ(X)\to M_{0,n}
\]
are $T$-torsors. We keep this assumption to avoid discussing tropicalizations of torus torsors and compactifications thereof. See also~\cite[Section 3]{CMR14b}.
\end{remark}

\section{Gromov--Witten invariants}

Given a linear subspace $L_i$ of $N_\QQ$, let $\mathbb G(L_i)$ be the associated subtorus. Translates of $\mathbb G(L_i)$ are parametrized by the quotient torus $\mathbb G(N/L_i)$. Observe that $\trop(\mathbb G(L_i)) = L_i\otimes \RR$ and by functoriality of tropicalization, the tropicalization of a translate $t_0\cdot \mathbb G(L_i)$ is the affine space obtained by translating $L_i$ by $\trop(t_0)\in \Hom(M,\RR)$. Let $Z_i$ be the closure of $\mathbb G(L_i)$. The correspondence theorem that we prove in this section will count rational curves in $X$ with given incidence to closures of linear subspaces. In order to use evaluation morphisms, we will need a compact space of $\mathbb G(L_i)$-torus orbits. In fact, we may choose any toric compactification of the quotient torus $T/\mathbb G(L_i)$ to which $X$ maps. The natural choice seems to be the Chow quotient, see Remark~\ref{whyChow} below.

\subsection{Counting curves in $X$} Given a marked point with trivial contact order, we have a map $ev: \Lsm_\Gamma(X)\to X$. By composing this with the Chow quotient map $X\to  X\!\sslash\!\mathbb G(L_i)$ we obtain, for each point $p_i$ and subspace $L_i$, evaluation maps
\[
{ev}_{L_i}: X\to X_{L_i}: = X\!\sslash\!\mathbb G(L_i).
\]
Putting these evaluation maps together we have
\[
{Ev}_{\mathscr L}:\Lsm_\Gamma(X)\to \prod_{i=1}^m X_{L_i}. 
\]
The logarithmic Gromov--Witten invariant is given by
\[
\langle Z_1,\ldots, Z_m\rangle^{X}_{\Gamma} = \deg(Ev_{\mathscr L}^\star(pt))
\] 

\begin{remark}\label{whyChow}
Note that since the map $\Lsm_\Gamma(X)\to \mathfrak M_{0,n+m}$ is logarithmically unobstructed, the logarithmic Gromov--Witten invariant is enumerative. In particular, the invariant above counts maps from $\PP^1$ to $X$ such that the marked point $p_i$ maps to the locally closed subscheme $\mathbb G(L_i)\subset X$. This justifies our use of the Chow quotient as an evaluation space.
\end{remark}

Let $ev_i:\Lsm_\Gamma(X)\to X$ be a logarithmic evaluation morphism at a marking $p_i$ carrying trivial contact order.

\begin{proposition}\label{prop: trop-analytic}
There following diagram commutes:
\[
\begin{tikzcd}
\Lsm^{\an}_\Gamma(X) \arrow{r}{ev_i^{\an}} \arrow{d}[swap]{\trop} & X^{\an} \arrow{d}{\trop} \\
\Tsm_\Gamma(\Delta) \arrow[swap]{r}{ev_i^{\trop}} & \overline\Delta.
\end{tikzcd}
\]

\end{proposition}

\begin{proof}
This follows from functoriality for formation of skeletons for logarithmic morphisms~\cite[Theorem 1.1]{U13} and Theorem~\ref{thm: skeleton}.
\end{proof}

\subsection{Proof of Theorem~\ref{thm: enumerative}} 

We will need the following fact about tropical multiplicities in order to deduce the correspondence theorem. The degree of the evaluation morphism above can be computed at the level of Berkovich analytic spaces, and hence at the level their skeletons. The lemma below ensures that there are no corrections in passing from the skeleton to the tropicalization.

\begin{lemma}\label{lem: multiplicity-1}
All tropical multiplicities of $\Tsm_\Gamma(\Delta)$ are equal to $1$.
\end{lemma}

\begin{proof}
By Proposition~\ref{prop: a-gross-proposition}, the tropicalization $\Lsm_\Gamma^\circ(X)$ coincides with the tropicalization of the space $M_{0,n+m}\times T$. Since $\overline M_{0,n+m}$ is a wonderful compactification of a hyperplane complement. It follows that all initial degenerations are smooth and irreducible, so all multiplicities are equal to $1$, see~\cite[Section 6.7]{MS14}.
\end{proof}

We now prove Theorem~\ref{thm: enumerative}. To reduce the burden of notation, let $\cE$ denote the fan of the toric variety $\prod_{i=1}^m X_{L_i}$. It follows from~\cite[Proposition 3.6]{Gro14} that $Ev_{\mathscr L}: \Lsm_\Gamma(X)\to  X(\cE)$ is the restriction of an morphism of toric varieties $X(\cT)\to X(\cE)$. As discussed in Section~\ref{sec: background-1}, since $X(\cE)$ is a toric variety, there is a canonical projection to the skeleton
\[
\bm p': X(\cE)^{\an}\to \overline \cE.
\] 
Let $\sigma$ be a top dimensional cell in $\cE$, and $\mathscr U$ be its preimage under $\bm p'$. The set $\mathscr U$ is a polyhedral affinoid domain in $T^{\an}\subset X^{\an}$. By~\cite[Lemma 8.3]{BPR} the degree of $Ev_{\mathscr L}$ is  the degree of the map
\[
(Ev_{\mathscr L}^{\an})^{-1}(\mathscr U)\to \mathscr U.
\]
By~\cite[Proposition 4.8]{GKM07} the map
\[
\Tsm^\circ_\Gamma(\Delta)\to \cE
\] 
is a morphism of fans. Since all tropical multiplicities are equal to $1$ by Lemma~\ref{lem: multiplicity-1} above. By the Sturmfels--Tevelev multiplicity formula~\cite[Theorem 1.1]{ST}, the absolute value of the determinant of the map $\trop(Ev_{\mathscr L}): \Tsm^\circ_\Gamma(\Delta) \to \cE$ calculates the degree of $Ev_{\mathscr L}$ and the result follows. \qed

\begin{remark}
The multiplicity obtained in the above result coincides with the one obtained by Mikhalkin~\cite{Mi03} for rational curves in $\PP^2$, where the approaches overlap, see~\cite[Remark 5.2]{GKM07}.
\end{remark}

\bibliographystyle{siam}
\bibliography{ToricLogMaps}

\end{document}